\documentclass[12pt, reqno]{amsart}
\makeatletter
\DeclareMathAlphabet{\mathpzc}{OT1}{pzc}{m}{it}
\usepackage[margin=2cm]{geometry}
\usepackage{amsmath, amssymb, amsfonts, amstext, verbatim, amsthm, mathrsfs}
\usepackage[mathcal]{eucal}
\usepackage{microtype}

\usepackage[all,arc,knot,poly]{xy}
\usepackage{enumerate}
\usepackage{mathrsfs}
\usepackage{mathtools}
\usepackage{amsxtra}
\usepackage{needspace}
\usepackage{graphicx}
\usepackage{setspace}

\usepackage[usenames]{color}
\usepackage{aliascnt} 
\usepackage[colorlinks=true,linkcolor=blue,citecolor=blue,urlcolor=blue,citebordercolor={0 0 1},urlbordercolor={0 0 1},linkbordercolor={0 0 1}]{hyperref} 
\usepackage{enumerate}
\usepackage{xspace}

\theoremstyle{plain}

\newcommand{\refnewtheoremn}[4]{
\newaliascnt{#1}{#2}
\newtheorem{#1}[#1]{#3}
\aliascntresetthe{#1}
\expandafter\providecommand\csname #1autorefname\endcsname{#4}}

\newcommand{\refnewtheorem}[3]{\refnewtheoremn{#1}{#2}{#3}{#3}}

\def\makeCal#1{
\expandafter\newcommand\csname c#1\endcsname{\mathcal{#1}}}
\def\makeBB#1{
\expandafter\newcommand\csname b#1\endcsname{\mathbb{#1}}}
\def\makeFrak#1{
\expandafter\newcommand\csname f#1\endcsname{\mathfrak{#1}}}

\count@=0
\loop
\advance\count@ 1
\edef\y{\@Alph\count@}
\expandafter\makeCal\y
\expandafter\makeBB\y
\expandafter\makeFrak\y
\ifnum\count@<26
\repeat

\newtheorem{thm}{Theorem}[section]
\newtheorem{theorem}{Theorem}[section]

\DeclareMathOperator{\sgn}{\operatorname{sgn}}

\refnewtheorem{lemma}{thm}{Lemma}
\refnewtheorem{claim}{thm}{Claim}
\refnewtheorem{cor}{thm}{Corollary}
\refnewtheorem{conj}{thm}{Conjecture}
\refnewtheorem{prop}{thm}{Proposition}
\refnewtheorem{proposition}{thm}{Proposition}
\refnewtheorem{quest}{thm}{Question}

\refnewtheorem{assumption}{thm}{Assumption}
\refnewtheorem{problem}{thm}{Problem}

\theoremstyle{definition}
\refnewtheorem{remark}{thm}{Remark}
\refnewtheorem{remarks}{thm}{Remarks}
\refnewtheorem{defn}{thm}{Definition}
\refnewtheorem{definition}{thm}{Definition}
\refnewtheorem{notn}{thm}{Notation}
\refnewtheorem{const}{thm}{Construction}
\refnewtheorem{example}{thm}{Example}
\refnewtheorem{examples}{thm}{Examples}
\refnewtheorem{fact}{thm}{Fact}

\newcommand {\id}{\operatorname{id}}

\renewcommand{\Im}{\operatorname{Im}}
\renewcommand{\Re}{\operatorname{Re}}
\newcommand {\Hom}{\operatorname{Hom}}
\newcommand {\Aut}{\operatorname{Aut}}

\newcommand{\SL}{\operatorname{SL}}
\newcommand{\tr}{\operatorname{tr}}

\newcommand{\GL}{\operatorname{GL}}

\newcommand{\Tri}{\operatorname{Tri}}
\newcommand{\diag}{\operatorname{diag}}

\newcommand{\PGL}{\operatorname{PGL}}

\newcommand{\MCG}{\operatorname{MCG}}

\newcommand{\Proj}{\operatorname{\cP}}

\newcommand{\res}{\operatorname{res}}

\newcommand{\Zer}{\mathrm{Zer}}
\newcommand{\Pol}{\mathrm{Pol}}
\newcommand{\Crit}{\operatorname{Crit}}

\newcommand{\onto}{\twoheadrightarrow}
\newcommand{\lra}{\longrightarrow}
\newcommand{\into}{\hookrightarrow}

\newcommand{\lRa}[1]{\stackrel{#1}{\lra}}

\newcommand{\isom}{\cong}

\newcommand{\tensor}{\otimes}
\renewcommand{\O}{\mathscr{O}}

\newcommand{\mat}[4]{\begin{pmatrix}#1&#2\\#3&#4\end{pmatrix}}
\newcommand{\vect}[2]{\begin{pmatrix}#1\\#2\end{pmatrix}}

\newcommand{\constant}{\operatorname{constant}}

\newcommand{\Grp}{\operatorname{Grp}}
\newcommand{\Set}{\operatorname{Set}}

\renewcommand{\varepsilon}{v}

\begin{document}

\title{The monodromy of meromorphic projective structures}
\author{Dylan G.L. Allegretti}
\author{Tom Bridgeland}

\date{}

\begin{abstract}{We study projective structures on a surface having poles of prescribed  orders. We obtain a monodromy map from a complex manifold parameterising such structures to the stack of  framed $\PGL_2(\bC)$ local systems on the associated marked bordered surface.  We prove that the image of this map is contained in the union of the domains of the cluster charts. We discuss a number of open questions concerning this monodromy map.} \end{abstract}

\maketitle

\section{Introduction}

This paper is concerned with the monodromy of projective structures on Riemann surfaces. A projective structure  can be viewed as a global generalization of a differential equation of the form
\[
\label{schro}y''(z) -\varphi(z) \cdot y(z)=0,
\]
where primes denote differentiation with respect to $z$.
In another language it is an oper for the group $\PGL_2(\bC)$.
Our main focus will be on the case when the potential $\varphi(z)$ has poles, but we begin by recalling some of the classical results on holomorphic projective structures. For excellent surveys on this material we recommend \cite{Dumas,LM,Tyurin}.

\subsection{Holomorphic projective structures}

A projective structure $\cP$ on a Riemann surface $S$ is an atlas of holomorphic charts
\[f_i\colon U_i\to \bP^1,  \qquad S=\bigcup_{i\in I} U_i,\]
with the property that each transition function  $g_{ij}=f_i\circ f_j^{-1}$ is the restriction of an element of \[G=\Aut(\bP^1)=\PGL_2(\bC).\]  

Projective structures  are closely related to quadratic differentials.  If $\cP$ is a projective structure on a Riemann surface $S$, with a local chart $z\colon U\to \bP^1$, and \begin{equation}
\label{qu}\phi(z)=\varphi(z)\, dz^{\tensor 2}\end{equation}
is a quadratic differential on $S$, written in terms of the co-ordinate $z$, then one obtains a chart in a new projective structure $\cP'=\cP+\phi$ on the surface $S$ by considering the ratio of two independent solutions to the differential equation
\begin{equation}
\label{schrod}y''(z) -\varphi(z) \cdot y(z)=0.\end{equation}
This construction gives the set of projective structures on a fixed compact Riemann surface $S$ of genus $g=g(S)$ the structure of an affine space for the vector space of holomorphic quadratic differentials \begin{equation}
\label{spacc}H^0(S,\omega_S^{\tensor 2})\isom \bC^{3g-3}.\end{equation} 

Let us fix a closed, oriented surface $\bS$ of genus  $g=g(\bS)$. A marked projective structure  is then defined to be a  triple $(S,\cP,\theta)$, where $S$ is a Riemann surface equipped with a projective structure~$\cP$, and $\theta$ is a marking, that is,  an isotopy class of orientation-preserving diffeomorphisms $\theta\colon \bS\to S$. Two such triples $(S_i,\cP_i,\theta_i)$ will be considered to be equivalent if there is a biholomorphism $f\colon S_1\to S_2$ which preserves the projective structures and  the markings in the obvious way.

The set  $\cP(\bS)$ of equivalence classes of marked projective structures   has the natural structure of a complex manifold of dimension $6g-6$. 
There is an obvious forgetful map
\begin{equation}\label{forge}p\colon \cP(\bS)\to \cT(\bS)\end{equation}
to the Teichm{\"u}ller space $\cT(\bS)$, viewed as the moduli space  of Riemann surfaces $S$ equipped with a marking $\theta\colon \bS\to S$.
A relative version of the construction described above shows that this  map \eqref{forge} is an affine bundle for the vector bundle 
\[\label{forger}q\colon \cQ(\bS)\to \cT(\bS),\] whose fibre over a marked Riemann surface $(S,\theta)$ is the vector space  \eqref{spacc}.

\subsection{Monodromy of projective structures}
\label{sni}

A projective structure  $\cP$ on a Riemann surface $S$ has an associated developing map: a holomorphic map \[f\colon \tilde{S}\to \bP^1,\] where $\pi\colon \tilde{S}\to S$ is the universal covering surface, such that  any injective locally-defined map of the form $f\circ \pi^{-1}$ is a chart in $\cP$. Such a  developing map  gives rise to a monodromy representation
\begin{equation}
\label{rep1}\rho\colon \pi_1(S)\to G,\end{equation}
defined by the the condition $f(\gamma\cdot x)=\rho(\gamma)\cdot  f(x)$. The developing map $f$  is unique up to post-composition with an element of the group $G$, and the monodromy representation is therefore well-defined up to overall conjugation by an element of $G$. More abstractly, we can think in terms of a $G$ local system naturally associated to the projective structure.

Let us fix a closed, oriented surface $\bS$ as above and consider the  quotient stack \begin{equation}
\label{cs}\cX(\bS)=\Hom(\pi_1(\bS),G)/G\end{equation}
parameterising representations \eqref{rep1} up to overall conjugation, or equivalently, isomorphism classes of $G$ local systems on $\bS$. The monodromy of a marked projective structure on $\bS$ defines a point of this stack in the obvious way. Gunning \cite{Gun} proved that when $g=g(\bS)>1$, this point  lies in the open substack $\cX^*(\bS)$ consisting of representations with non-commutative image. This substack  is a (possibly non-Hausdorff) complex manifold, and there is a  holomorphic map 
\begin{equation}\label{monmap}F\colon \cP(\bS)\to \cX^*(\bS)\end{equation}
sending a marked projective structure to its monodromy representation.

The map $F$ has been studied for over a century. Let us briefly recall some of the better known results. A result of Poincar{\'e} \cite{Po} shows that $F$ is injective when restricted to each fibre of the forgetful map \eqref{forge}. Hejhal \cite{Hejhal} proved that $F$ is a local homeomorphism, and Earle \cite{Earle} and Hubbard~\cite{Hub} showed that $F$ is moreover a local biholomorphism. It is also known that $F$ has infinite fibres and is not a covering map of its image, see e.g. \cite{Dumas}. Finally, a famous theorem of Gallo, Kapovich and Marden \cite{GKM} characterises the image of  $F$.

\subsection{Meromorphic projective structures}
In  this paper we study a monodromy map analogous to \eqref{monmap}  but for meromorphic projective structures. 
 The notion of a meromorphic projective structure has meaning because of the above-mentioned fact  that the set of projective structures on a fixed Riemann surface $S$ is an affine space for the space of quadratic differentials. 
Concretely, if we fix a holomorphic projective structure on $S$ as above, the local charts in a meromorphic projective structure  are obtained by taking ratios of solutions to  the equation  \eqref{schrod}, where  the quadratic differential \eqref{qu} is now allowed to have poles. 

Note that when $\varphi(z)$ has a pole $p$ of order $m >2$ the equation \eqref{schrod}  has an irregular singularity, and one should expect generalised monodromy in the form of Stokes data to enter the picture. A  solution $y(z)$ to the equation \eqref{schrod} defined in a sectorial neighbourhood centered at $p$  is called subdominant if $y(z)\to 0$ as $z\to p$. Standard results in the theory of differential equations show that there are $m-2$ distinguished sectors centered at $p$, known as Stokes sectors,  in which there exist unique-up-to-scale subdominant solutions to \eqref{schrod}. The rays in the centres of these distinguished sectors are called Stokes directions and coincide with the asymptotic directions at $p$ of the horizontal foliation defined by the quadratic differential $\phi(z)$.

Our results depend on a choice of genus $g\geq 0$ and a non-empty collection of positive integers  giving the orders of poles of the projective structures. It is more convenient to represent this data in the form of a marked bordered surface $(\bS,\bM)$. This is a compact, connected, oriented  surface  $\bS$ with (possibly empty) boundary, equipped with a non-empty collection of marked points $\bM\subset \bS$, such that each boundary component of $\bS$ contains at least one point of $\bM$. We denote by $\bP\subset \bM$ the set of internal marked points, which we also refer to as punctures.

A meromorphic projective structure $\cP$ on a compact Riemann surface $S$ naturally determines a marked bordered surface $(\bS,\bM)$. The surface $\bS$ is the real oriented blow-up of $S$ at the  poles of $\cP$ of order $>2$, which are precisely the irregular singularities of the differential equation \eqref{schrod}, and  the boundary marked points correspond to  the Stokes directions. The internal marked points  are the poles of $\cP$ of order  $\leq 2$, which are precisely the regular singularities of the equation \eqref{schrod}.

Let us  fix a marked bordered surface $(\bS,\bM)$. If $g(\bS)=0$  we assume that $|\bM|\geq 3$. In Section \ref{msps} we introduce a complex manifold $\Proj(\bS,\bM)$ parameterising marked meromorphic projective structures in much the same way as before. The points of $\Proj(\bS,\bM)$ are equivalence-classes of triples $(S,\cP,\theta)$, where $S$ is a Riemann surface equipped with a meromorphic projective structure $\cP$, and $\theta$ is an isotopy class of orientation-preserving diffeomorphisms between the marked bordered surface canonically associated to $(S,\cP)$ and the fixed surface $(\bS,\bM)$.

It will be convenient to  introduce two modifications to the manifold $\Proj(\bS,\bM)$. Firstly,  we define a dense open subset \[\Proj^\circ(\bS,\bM)\subset \Proj(\bS,\bM),\] whose complement consists of projective structures with apparent singularities: regular singularities for which  the corresponding monodromy transformation is trivial as an element of $G=\PGL_2(\bC)$.  We do this because  the monodromy map of Theorem \ref{one} is not well-defined for such projective structures. Secondly, we introduce a branched cover\begin{equation}
\label{coverproj1}\pi\colon \Proj^*(\bS,\bM)\to \Proj^\circ(\bS,\bM)\end{equation}    of  degree $2^{|\bP|}$,  whose points parameterise marked projective structures equipped with a signing: a choice of eigenline for the monodromy around each regular singularity. 
The main point of  this is to obtain a monodromy map taking values in the stack of framed local systems, which as we explain below, is rational and carries interesting birational co-ordinate systems. An analogous cover also turns out to be very natural in the context of moduli spaces of meromorphic quadratic differentials \cite[Section 6.2]{BS}.

\subsection{Monodromy and framed local systems}

 Let us now turn to the analogue of the character stack \eqref{cs} in the meromorphic situation. A framed  $G=\PGL_2(\bC)$ local system  on a marked bordered surface $(\bS,\bM)$ is defined to be a $G$ local system $\cG$ on the punctured  surface $\bS^*=\bS\setminus\bP$ equipped with the data of a framing: a choice of flat section $\ell(p)$ of the restriction of the associated $\bP^1$-bundle
\[\cL=\cG\times_G \bP^1\]
to a neighbourhood of each point $p\in \bM$. The  moduli stack $\cX(\bS,\bM)$ of  framed $G$ local systems on $(\bS,\bM)$ was introduced by Fock and Goncharov in \cite{IHES}.

We shall call a framed local system \emph{degenerate} if one of the following  statements holds:
\begin{itemize}
\item[(D1)] There is a connected component $I$ of the punctured boundary $\partial \bS\setminus \bM$ such that the two chosen sections $\ell(p_i)$ defined near the points $p_1,p_2\in \bM$ lying at the ends of  $I$    coincide under parallel transport along $I$.
\smallskip

\item[(D2)]
There exists an unordered  pair of distinct points  $\ell_\pm(p)$ in each fibre of the bundle $\cL$, each of which locally defines a flat section of $\cL$ but which may be permuted by the monodromy of $\cL$,   such that for each $p\in \bM$ the chosen section $\ell(p)$ coincides with one of the $\ell_\pm(p)$. \smallskip

\item[(D3)]  The surface $\bS$ is closed, there  exists a single flat section $\ell$ of the bundle $\cL$ over $\bS^*$ such that for each $p\in \bM$ the chosen section $\ell(p)$ coincides with $\ell$,   and moreover, the monodromy around each puncture is either parabolic or the identity.
 \end{itemize}
  
Note that for any given surface $(\bS,\bM)$ we need only ever consider two of the above conditions, since when $\bS$ is  closed  condition (D1) is vacuous.
We show that there is an open substack \[\cX^*(\bS,\bM)\subset \cX(\bS,\bM)\] parameterising non-degenerate framed local systems, and  use  Fock-Goncharov co-ordinates (see Section \ref{cluco} below) to show that it is a (possibly non-Hausdorff) complex manifold. 
 
 A meromorphic projective structure on a Riemann surface $S$ induces a holomorphic projective structure on the complement  of its poles, and hence a monodromy $G$ local system on the associated punctured surface $\bS^*$. We show that for signed projective structures without apparent singularities this $G$ local system has a natural framing. At regular singularities, this   is provided by the choice of signing. At irregular singularities it  is provided by the unique up-to-scale subdominant solutions defined on Stokes sectors. We call the resulting framed local system the monodromy framed local system of the signed projective structure.
 
 In Section \ref{locale} we prove that the monodromy framed local system of a signed projective structure without apparent singularities is non-degenerate. This leads to the following result.

\begin{thm}
\label{one}
Let $(\bS,\bM)$ be a marked bordered surface, and if $g(\bS)=0$ assume that $|\bM|\geq 3$.  Then there is a holomorphic map
\begin{equation}
\label{map}F\colon \Proj^*(\bS,\bM)\to \cX^*(\bS,\bM),\end{equation}
sending a signed marked projective structure to its  monodromy  framed $G$ local system.
\end{thm}

The hypothesis in the statement of  Theorem \ref{one} excludes six marked bordered surfaces. We discuss these degenerate cases in more detail in Section \ref{remark}. In the two cases when $g(\bS)=0$ and $|\bM|=1$, both sides of \eqref{map} are easily seen to be empty, so the result is vacuously true. The remaining four cases, when $g(\bS)=0$ and $|\bM|=2$, can be treated directly: in two of them the  statement of Theorem \ref{one} needs to be modified to account for the fact that the  space $\Proj^*(\bS,\bM)$ is no longer a manifold, since it has non-trivial generic automorphism group. 

\subsection{Cluster co-ordinates}
\label{cluco}
One very important and interesting feature of the meromorphic situation, which has no analogue in the holomorphic case, is that the stack $\cX(\bS,\bM)$ of framed local systems, when considered in the algebraic category,  is rational. 
To prove this, Fock and Goncharov~\cite{IHES} constructed an explicit  system of birational maps \[X_T\colon \cX(\bS,\bM)\dashrightarrow (\bC^*)^n\] indexed by the ideal triangulations $T$ of the surface $(\bS,\bM)$. Moreover, they proved that the inverse birational maps induce  open embeddings
\[X_T^{-1}\colon (\bC^*)^n \into \cX(\bS,\bM).\]

In fact, when the surface $(\bS,\bM)$ has punctures, the correct combinatorial object to consider  is not an ideal triangulation, but a variant known as a tagged triangulation \cite{FST}. The set of tagged triangulations contains the set of ideal triangulations within it, but has the additional feature that the graph whose vertices are tagged triangulations and whose edges are flips is always regular. It is easy to  extend Fock and Goncharov's
definition to give birational maps \[X_\tau\colon \cX(\bS,\bM)\dashrightarrow (\bC^*)^n\] indexed by tagged triangulations $\tau$ of the surface $(\bS,\bM)$. Note however that for  general tagged triangulations, the inverse map $X_\tau^{-1}$ need not be regular on the whole algebraic torus $(\bC^*)^n$.

We say that a point of $ \cX(\bS,\bM)$ is generic with respect to a tagged triangulation $\tau$ if it lies in the domain of the map $X_\tau$.
In Section \ref{final} we  use a combinatorial argument to prove the following result.

\begin{thm}
\label{extra}
Any point of the open substack $\cX^*(\bS,\bM)$  is generic with respect to some tagged triangulation. If the surface $\bS$ has non-empty boundary this tagged triangulation can be taken to be an ideal triangulation. \comment{Do we need to exclude the three-punctured sphere?}
\end{thm}

Associated to a marked bordered surface $(\bS,\bM)$ is a space  known  as the cluster Poisson variety. It is defined as an abstract union of algebraic tori
\[\cX_{cl}(\bS,\bM)=\bigcup_{\tau} \, (\bC^*)^n\]
 indexed by the tagged triangulations of $(\bS,\bM)$, glued together with explicit birational transformations corresponding to flips of tagged triangulations. It can be viewed as a (possibly non-separated) smooth scheme, or as a (possibly non-Hausdorff) complex manifold. The birational maps $X_\tau$ induce a birational map
\[\cX(\bS,\bM)\dashrightarrow \cX_{cl}(\bS,\bM),\]
which is well-defined on the open subset of points which are generic with respect to some  tagged triangulation $\tau$. Thus, as a consequence of Theorem \ref{extra}, we can alternatively view the map $F$ as taking values in the cluster Poisson variety $\cX_{cl}(\bS,\bM)$. 

 \subsection{Relation to existing works}

Several special cases of the map of Theorem \ref{one} have been studied before in the literature. We discuss these below, along with some other related works.

\subsubsection{Work of Bakken and Sibuya}\label{bakks}

In the case when $(\bS,\bM)$ is a disc with $m+3$ marked points on the boundary,  the space $\Proj(\bS,\bM)$ parameterises marked projective structures on $\bP^1$ with a single pole of order $m+5$. The monodromy map of Theorem \ref{one}   was studied in detail in this case by Sibuya  and Bakken.  The space $\Proj(\bS,\bM)$  can be identified with the space of polynomials
\[\varphi(z)=z^{m+1}+a_{m-1} z^{m-1} + \cdots +a_1 z + a_0,\]
with the corresponding projective structure being given by ratios of linearly independent solutions to the equation \eqref{schrod}. The
stack $\cX(\bS,\bM)$ parameterises  cyclically ordered $(m+3)$--tuples of points of $\bP^1$ up to the diagonal action of $G=\PGL_2(\bC)$. The open substack $\cX^*(\bS,\bM)$ consists of tuples having no two consecutive points equal and containing at least three distinct points. The results of Sibuya \cite{Sibuya75} and Bakken \cite{Bakken} imply the following.

\begin{thm} Suppose that $(\bS,\bM)$ is an unpunctured disc with $\geq 3$ marked points on the boundary. Then the map $F$ of Theorem \ref{one} is a local biholomorphism, and is moreover surjective.
\end{thm}

Sibuya's proof of the surjectivity relies on results of Nevanlinna \cite{Nev}. Note that in these cases the cluster variety $\cX_{cl}(\bS,\bM)$ coincides with the space of non-degenerate framed local systems $\cX^*(\bS,\bM)$.  More details on these examples can be found in \cite{A}.

\subsubsection{Work of  Iwasaki and Luo}
Suppose that   $(\bS,\bM)$ is a marked bordered surface such that $\bS$ is closed. The marked points $\bM$ are then all punctures, and the space $\Proj(\bS,\bM)$ parameterises meromorphic projective structures with regular singularities.
The character stack $\cX(\bS^*)$ is defined exactly as in Section \ref{sni}, simply replacing the compact surface $\bS$ with the punctured surface $\bS^*$.  Rather than considering the cover \eqref{coverproj1}, one can also consider the monodromy map
\[F\colon \Proj(\bS,\bM)\to \cX(\bS^*),\]
which sends a projective structure to its monodromy representation.  This map  has been considered by Iwasaki \cite{Iwasaki1, Iwasaki2} and Luo \cite{FL} among other authors. Luo proved that $F$
is a local biholomorhism near any projective structure whose puncture-monodromies are non-parabolic, and whose monodromy representation is a smooth point of the stack $\cX(\bS^*)$. Iwasaki considered   projective structures with fixed numbers of apparent singularities and also obtained a local biholomorphism result.

\subsubsection{Wild character variety and the irregular Riemann-Hilbert map}Given an arbitrary connected reductive group $G$, Boalch \cite{Boalch2} introduced a wild character variety which is the natural receptacle for the generalised monodromy of $G$ connections with irregular singularities on a compact Riemann surface. The case of ramified leading order terms is treated in  \cite{BY}. Our  space of framed local systems is presumably closely related to the twisted character variety for the group $G=\PGL_2(\bC)$, although it seems not to be exactly the same, and we do not have a precise understanding of the relationship between these two spaces. Irregular character varieties have also been considered in the context of cluster algebras in \cite{CMR}.

In any case, the map $F$ of Theorem \ref{one} should certainly not be confused with the irregular Riemann-Hilbert map, sending a meromorphic connection to its monodromy, even if the targets of these maps are closely related. We deal not with arbitrary flat connections but with opers: in other words, with differential equations rather than differential systems. On the other hand, we allow the complex structure on our Riemann surface $S$ to vary. In contrast, the Riemann-Hilbert correspondence deals with arbitrary flat connections, but on a fixed Riemann surface~$S$. These two differences play in opposite directions, with the result that the appropriate space of flat $\PGL_2(\bC)$ connections on a fixed Riemann surface $S$ has the same dimension as our space of projective structures on the fixed topological surface $\bS$.

\subsection{Further directions}

There are several interesting open questions concerning the map  of Theorem \ref{one}. There are also some interesting lines of further research, some of which we plan to return to in future publications.

\subsubsection{Analogues of Hejhal and  Gallo-Kapovich-Marden Theorems}

It is natural to conjecture that the analogue of Hejhal's Theorem holds, and that the map $F$ of Theorem \ref{one} is always a 
local biholomorphism.
As noted  above, this is known in the case when $(\bS,\bM)$ is an unpunctured disc. In the case when  $\bS$ is a closed surface, the works of Iwasaki \cite{Iwasaki1} and  Luo \cite{FL}   mentioned above also give some supporting evidence.

Another natural question is whether  an analogue  of the Gallo-Kapovich-Marden result exists in the context of Theorem \ref{one}. Note that in the case of the unpunctured disc case the image of~$F$ is precisely the subset $\cX^*(\bS,\bM)$. Is it possible that this holds for an arbitrary unpunctured surface? The results of~\cite{BMM} could be relevant here. In the presence of punctures one probably should not expect a nice characterisation of the image of~$F$, since  in this case we have excluded projective structures with apparent singularities from the domain of the map. It could  be  interesting however to try to understand the behaviour of the map at these points: for example, does $F$ extend holomorphically to the blow-up of $\cP(\bS,\bM)$ along the codimension two locus of projective structures with apparent singularities?

\subsubsection{Exact WKB analysis and Voros symbols}

The space of meromorphic projective structures on a fixed Riemann surface $S$ is an affine space for the vector space of meromorphic quadratic differentials on $S$. Consider a family of projective structures of the form \[\cP(\hbar)=\cP+\hbar^{-2}\cdot  \phi,\] with $\cP$ a base meromorphic projective structure, $\phi$ a meromorphic quadratic differential, and $\hbar \in \bC^*$. Taking a local chart $z$ in the projective structure $\cP$ we can write $\phi(z)=\varphi(z)\, dz^{\tensor 2}$, and the local charts for $\cP(\hbar)$  are then given by ratios of solutions to the equation
\begin{equation}
\label{hschrod2}
\hbar^2\, y''(z)-\varphi(z) \cdot y(z)=0.\end{equation}
Let us assume that in the terminology of \cite{BS} the differential $\phi$ is complete and saddle-free. This means that it has no simple poles and no finite-length horizontal trajectories. As explained in \cite{BS,GMN2}, the trajectory structure then determines a  triangulation $T(\phi)$ of the associated marked bordered surface $(\bS,\bM)$.

The behaviour of solutions to  equations of the form \eqref{hschrod2} for small values of $\hbar$ is the subject of WKB analysis and has been much-studied since the advent of quantum mechanics. 
An important concept in the modern theory of WKB analysis is the notion of a Voros symbol \cite{IwakiNakanishi,DDP93}. The Voros symbols are initially defined as formal sums in $\hbar$, but under the above saddle-free condition, their Borel sums  define analytic functions for  small and positive values of $\hbar$.  This crucial result follows from the forthcoming work \cite{KoSc}. The resulting analytic Voros symbols have a very natural interpretation in terms of our monodromy map $F$:  they are the Fock-Goncharov co-ordinates with respect to the  triangulation $T(\phi)$ of the image of the projective structure $\cP(\hbar)$ under the monodromy map $F$. (To be  more precise, one should choose a signing of the quadratic differential $\phi$ as in \cite{BS} and take Fock-Goncharov co-ordinates with respect to the corresponding signed or tagged version of the triangulation $T(\phi)$.) This insight is originally due,  in a slightly different context, to Gaiotto, Moore, and Neitzke \cite{GMN2}. Details will appear in an accompanying paper \cite{A2}.

\subsubsection{Families of reference projective structures}
We explained above that the forgetful map  \[p\colon \cP(\bS)\to \cT(\bS)\] is an affine bundle for the bundle of quadratic differentials 
\[q\colon \cQ(\bS)\to \cT(\bS)\]over Teichm{\"u}ller space. In fact, since Teichm{\"u}ller space is Stein, this affine bundle is necessarily trivial, and therefore  has holomorphic sections. We refer to such a holomorphic section as a family of reference projective structures. Once such a family has been chosen, one obtains a (completely non-canonical) identification
$\cP(\bS)\isom \cQ(\bS)$, and composing with the monodromy map gives a local biholomorphism
\[H\colon \cQ(\bS)\to \cX^*(\bS).\]
There are several natural choices of reference projective structures. Kawai \cite{Kaw} proved that for some of these the resulting map $H$ respects the natural Poisson structures on both sides. 

One can play the same games in the meromorphic case using the forgetful map which sends a point of the space $\cP(\bS,\bM)$ to the underlying marked Riemann surface together with the positions of the poles. This defines an affine bundle over the Teichm{\"u}ller space $\cT(g,d)$ of  marked Riemann surfaces of genus $g$ equipped with a collection of $d$ marked points. Once again, this bundle is trivial, so one can choose families of reference projective structures and obtain (again, totally non-canonical) identifications between $\cP(\bS,\bM)$ and a space $\cQ(\bS,\bM)$ parameterising marked Riemann surfaces equipped with meromorphic quadratic differentials. If the reference family is chosen appropriately, this identification lifts to structures equipped with signings, and  we obtain a holomorphic map
\begin{equation}
\label{bog}H\colon \cQ^\pm (\bS,\bM)\to \cX^*(\bS,\bM).\end{equation}
One can ask whether there are natural choices of reference projective structures for which the resulting map $H$ preserves the natural Poisson structures. The case of regular singularities, corresponding to a closed surface $\bS$, has been studied in \cite{Kor} extending results of \cite{BKN}.

\subsubsection{Stability conditions and cluster varieties}
Our interest in this subject arose from a more general  project involving spaces of stability conditions and cluster varieties. We will return to this topic in a future paper~\cite{A3}; the case of an unpunctured disc is treated in \cite{A}. Associated to a nondegenerate quiver with potential are two complex manifolds of the same dimension. The first is the space of  stability conditions~\cite{B}, and the second is the cluster Poisson variety. The structure of each space is controlled by the combinatorics of the exchange graph of the quiver, but the combinatorics is used quite differently in the two cases. Indeed, the space of stability conditions has a wall-and-chamber decomposition, whereas the cluster variety is defined as a union of algebraic tori glued by birational maps.

Ideas from theoretical physics, and in particular the work of Gaiotto, Moore, and Neitzke \cite{GMN1,GMN2}, suggest that these two spaces should be related in a highly non-trivial way, involving Donaldson-Thomas invariants and the Kontsevich-Soibelman wall-crossing formula. One interesting class of examples of quivers with potential arises from marked bordered surfaces $(\bS,\bM)$. Moreover, the space of stability conditions in these cases was shown in \cite{BS} to coincide with a moduli space $\operatorname{Quad}_{\heartsuit}(\bS,\bM)$ of Riemann surfaces equipped with meromorphic quadratic differentials. On the other hand, the cluster Poisson variety is closely related to the space $\cX(\bS,\bM)$ of framed local systems. We expect that the map \eqref{bog} fits into a much more general story relating stability spaces to cluster varieties, and that in general these maps can be constructed by solving the Riemann-Hilbert problems of \cite{RHDT}.

\subsection*{Acknowledgements. } We are very grateful  to Kohei Iwaki,  Dima Korotkin, Davide Masoero, and Tom Sutherland for useful conversations and correspondence. We are particularly grateful to Misha~Kapovich who pointed us to the reference \cite{Falt} which enabled us to extend our results to the case of projective structures with only regular singularities.

\section{Projective structures and opers}
\label{proj}

In this section we recall some basic definitions concerning projective structures on Riemann surfaces. This is all well-known material, which can be found for example in \cite{Dumas, Hub}. We also recall the relation with opers for the group $\PGL_2(\bC)$. This is less well-covered in the literature, but brief treatments can be found in \cite{BD,Fr} for example.

\subsection{Projective structures}
\label{sta}
We start by defining the notion of a projective structure. 

\begin{defn}Let $S$ be a Riemann surface.
\begin{itemize}\item[(i)] A \emph{chart} on $S$ is a biholomorphism $\phi\colon U\to V$ where $U\subset S$ and $V\subset \bP^1$ are non-empty open subsets.

\item[(ii)]Two charts $\phi_i\colon U_i\to V_i$  on $S$ are said to be \emph{projectively compatible} if 
\[\text{there exists } \mat{a}{b}{c}{d}\in \PGL_2(\bC) \text{ such that }\phi_2(x)=\frac{a\phi_1(x)+b}{c\phi_1(x)+d} \text{ for all } x\in U_1\cap U_2.\]
\item[(iii)]A \emph{projective atlas}  is a collection of projectively compatible charts whose domains cover $S$.
\item[(iv)] A \emph{projective structure} on $S$ is a maximal projective atlas.
\end{itemize}
\end{defn}

Given a projective structure $\cP$ on a Riemann surface $S$ we can consider the presheaf whose sections over an open subset $U\subset S$ is  the set of  charts of $\cP$  defined on $U$. The associated sheaf $\cG$ is a locally constant sheaf of $G$-torsors, which we refer to as the \emph{monodromy local system} of $\cP$. This gives rise to a \emph{monodromy representation}
\begin{equation}
\label{rep}\rho\colon \pi_1(S)\to G,\end{equation}
well-defined up to overall conjugation by an element of $G$, in the usual way. More concretely, one can  show (see, for example \cite[Lemma 1]{Hub}), that $\cP$ has a \emph{developing map}: a holomorphic map \begin{equation}
\label{devv}f\colon \tilde{S}\to \bP^1,\end{equation} where $\pi\colon \tilde{S}\to S$ is the universal covering surface, such that  any injective locally-defined map of the form $f\circ \pi^{-1}$ is a chart in $\cP$. Such a  developing map $f$  is unique up to post-composition with an element of $G$. The representation \eqref{rep} can then be defined by the  condition $f(\gamma\cdot x)=\rho(\gamma)\cdot  f(x)$ and once again is well-defined up to overall conjugation by an element of $G$.

\begin{example}
Let $S$ be a Riemann surface with universal cover $\pi\colon \tilde{S}\to S$. The uniformization theorem states that $\tilde{S}$ is biholomorphic to either the disc, the complex plane, or the Riemann sphere, and hence can be identified with an open subset $U\subset \bP^1$.  Moreover, the induced action of $\pi_1(S)$ by biholomorphisms on $\tilde{S}\isom U$ is then the restriction of an action on $\bP^1$. This gives rise to a projective structure on $S$: the charts are obtained by taking locally-defined inverses to $\pi$ and composing with the identification $\tilde{S}\isom U\subset \bP^1$.
\end{example}

\subsection{Affine space structure}

The following well-known result will be of fundamental importance in what follows.

\begin{prop}
\label{fun}The set of projective structures on a Riemann surface $S$ is an affine space for the vector space $H^0(S,\omega_S^{\tensor 2})$ of global quadratic differentials. 
\end{prop}

Thus given a projective structure $\cP_1$ and a quadratic differential $\phi\in H^0(S,\omega_S^{\tensor 2})$ there is a well-defined projective structure $\cP_2=\cP_1+ \phi$. Conversely, given two projective structures $\cP_1$ and $\cP_2$ there is a well-defined difference $\phi=\cP_2-\cP_1\in H^0(S,\omega_S^{\tensor 2})$. We briefly explain these constructions: for full details see \cite[\S 2]{Hub}.

\begin{itemize}
\item[(i)]Suppose given a projective structure $\cP_1$ and a quadratic differential $\phi\in H^0(S,\omega_S^{\tensor 2})$. Take a local co-ordinate $z\colon U\to \bC$ in the projective structure $\cP_1$ and write $\phi$ in the form $\varphi(z) dz^{\tensor 2}$. Then local charts $w\colon U\to \bP^1$ in the projective structure $\cP_2=\cP_1+ \phi$ are obtained by taking ratios $w(z)=y_1(z)/y_2(z)$ of two linearly independent  solutions of the differential equation 
\begin{equation}
\label{cat}
y''(z)-\varphi(z)\cdot y(z)=0,
\end{equation}
where primes denote differentiation with respect to $z$.

\item[(ii)]Conversely, given two projective structures $\cP_1$ and $\cP_2$, we can take local co-ordinates $z_i\colon U_i\to \bC$ with overlapping domains, and write $z_2=f(z_1)$ on the overlap. Then the quadratic differential $\phi=\cP_2-\cP_1$ is given on the overlap by the Schwarzian derivative
\begin{equation}
\label{schwarz}\phi=-\frac{1}{2}\cdot \Bigg[\bigg(\frac {f''}{f'}\bigg)'-\frac{1}{2}\bigg(\frac{f''}{f'}\bigg)^2 \Bigg]dz_1^{\tensor 2}.\end{equation}
It follows from basic properties of the Schwarzian derivative  that these expressions glue to give a global holomorphic differential on $S$.
\end{itemize}

\subsection{Opers for the groups $\GL_2(\bC)$ and $\PGL_2(\bC)$}
\label{pot}
The language of opers provides a useful alternative perspective on projective structures.  A \emph{$\GL_2(\bC)$ oper} on a  Riemann surface $S$ is defined to be triple $(E,L,\nabla)$, where $E$ is a rank two holomorphic vector bundle,
\[\nabla\colon E\to E\tensor \omega_S\]
is a flat holomorphic connection, and 
$L\subset E$ is a  line sub-bundle, such that
 the  $\O_S$-linear map $\theta$ defined as the composite\begin{equation}
 \label{theta} \theta\colon L\lRa{i} E \lRa{\nabla} E\tensor \omega_S \lRa{q\tensor \id_{\omega_S}} (E/L)\tensor \omega_S\end{equation}  is an isomorphism. Here $i\colon L\to E$ and $q\colon E\to E/L$ are the inclusion and quotient maps. 
It follows that $E$ fits into a short exact sequence
\[0\lra L\lra E\lra L\tensor \omega_S^*\lra 0.\]

Suppose given two $\GL_2(\bC)$ opers $(E_i,L_i,\nabla_i)$ on the same Riemann surface $S$. We consider them to be equivalent if there exists a line bundle with flat connection $(M,\nabla_{M})$ and an isomorphism
\[( E_1,L_1,\nabla_1)\isom ( E_2,L_2,\nabla_2)\tensor(M,\nabla_M),\]
in the sense that there is an isomorphism of vector bundles $ E_1 \to E_2\tensor M$ which takes the sub-bundle $L_1$ to the sub-bundle $L_2\tensor M$, and the connection $\nabla_1$ to the connection $\nabla_2\tensor 1_M + 1_{E_2}\tensor \nabla_M$. 
We define a \emph{$\PGL_2(\bC)$ oper}  to be an equivalence class of  of $\GL_2(\bC)$ opers under this relation.

\subsection{Opers and projective structures}
\label{plo}
There is a well-known correspondence between $\PGL_2(\bC)$ opers and projective structures. We briefly recall the relevant constructions here. 

\begin{itemize}
\item[(i)] Let $(E,L,\nabla)$ be a $\GL_2(\bC)$ oper on $S$. The projective bundle \[\pi\colon \bP(E)\to S\]
associated to $E$ is a $\bP^1$ bundle. The sub-bundle $L$ determines a section $\sigma(L)$ of this bundle. Locally, over an open subset $U\subset S$, one may choose a basis of flat sections of $\nabla$, and hence obtain an identification
\begin{equation}
\label{guff}\pi^{-1}(U) \isom \bP^1\times U,\end{equation}
with the trivial $\bP^1$ bundle over $U$. Under this identification the section $\sigma(L)$ becomes a holomorphic map $f\colon U\to \bP^1$, and it is easy to see that the oper condition implies that the derivative of $f$ is nowhere vanishing. Since the identification \eqref{guff} is uniquely defined up to the action of  $G$ on $\bP^1$, it follows that the charts $f$ obtained in this way glue  to give a projective structure on $S$.

\smallskip
\item[(ii)] Consider $\bP^1$ equipped with the standard action of $\GL_2(\bC)$. There is a  short exact sequence of vector bundles
\begin{equation}
\label{last}0\lra \O(-1)\lra \O^{\oplus 2}\lra \O(1)\lra 0.\end{equation}
The important point is that this is naturally an exact sequence of $\GL_2(\bC)$ equivariant vector bundles, where the action on the rank two trivial bundle is via the defining representation. One way to see this is to note that under the standard correspondence between graded modules for the polynomial ring and coherent sheaves on projective space, the sequence \eqref{last} is induced by the Koszul resolution.

Let us equip the trivial bundle in \eqref{last} with the canonical flat connection $d$. Then the triple \begin{equation}
\label{oper}(\O^{\oplus 2},\O(-1),d)\end{equation} is a $\GL_2(\bC)$ oper on $\bP^1$, which is equivariant with respect to the action of $\GL_2(\bC)$. Given a projective structure on a Riemann surface $S$, the monodromy representation \eqref{rep} can be lifted \cite{GKM} to a homomorphism \[\tilde{\rho}\colon \pi_1(S)\to \GL_2(\bC).\] This follows because of the identification with the monodromy action on solutions to the equation \eqref{cat}. Pulling back the oper \eqref{oper} on $\bP^1$ via the developing map \eqref{devv} then gives a $\GL_2(\bC)$ oper on the universal cover $\tilde{S}$ which is equivariant under the action of the fundamental group $\pi_1(S)$. This then descends to give a $\GL_2(\bC)$ oper on $S$.
\end{itemize}

It is not hard to check that the above correspondences define mutually-inverse bijections between projective structures and $\PGL_2(\bC)$ opers. Details can be found in \cite[Section 4.1]{Fr}.

\subsection{Affine structure of opers}
\label{ao}

Given a $\GL_2(\bC)$ oper $(E,L,\nabla)$ there is a linear map
\[\alpha\colon \Hom_S(\O_S,\omega_S^{\tensor 2}) \to  \Hom_S(E,E\tensor \omega_S),\]
obtained by sending a quadratic differential $\phi$ to the composite
\[\alpha(\phi): E\lRa{q} E/L \lRa{\theta^{-1}\tensor \id_{\omega_S^{-1}}} L\tensor \omega_S^{-1}\lRa{\id_{L\tensor \omega_S^{-1}}\tensor \phi}  L\tensor \omega_S \lRa{i\tensor \id_{\omega_S}} E\tensor \omega_S\]
where as before $i\colon L\to E$ and $q\colon E\to E/L$ are the inclusion and quotient maps. This leads to a linear action of the vector space $H^0(S,\omega_S^{\tensor 2})$ on the set of opers, in which a quadratic differential $\phi$ acts by
\[\phi\colon (E,L,\nabla)\mapsto (E,L,\nabla+\alpha(\phi)).\]
This action preserves the equivalence relation of Section \ref{pot} and hence descends to an action on the set of $\PGL_2(\bC)$ opers.

\begin{lemma}
\label{plob}
The bijection between opers and projective structures commutes with the  action of quadratic differentials on both sides. In particular  the set  of $\PGL_2(\bC)$ opers on $S$ becomes an affine space for the vector space $H^0(S,\omega_S^{\tensor 2})$ under the above action.
\end{lemma}

\begin{proof}
Take a $\GL_2(\bC)$ oper $(E,L,\nabla)$ and a local chart $z\colon U\to \bC$ in the corresponding projective structure. By construction of this projective structure we can find flat sections $s_1(z),s_2(z)$ of $E$ over $U$ such that the sub-bundle $L\subset E$ is spanned by the section $v(z)=s_1(z)+z s_2(z)$. Let us also write $u(z)=s_2(z)$. Then we have $\nabla_{\partial/\partial z}(v(z))=u(z)$ and $\nabla_{\partial/\partial z}(u(z))=0$, so in the local basis of sections $(u(z),v(z))$ the connection takes the form
\[\nabla=d+\mat{0}{1}{0}{0} dz.\]
By mild abuse of notation we can view $v(z)$ as a section of $L$ via the inclusion $L\subset E$, and $u(z)$ as a section of $E/L$ via the projection $E\onto E/L$. The bundle map $\theta$ then sends $v(z)$ to $u(z)\tensor dz$. 

Let us now choose a quadratic differential $\phi=\varphi(z)dz^{\tensor 2}$. The map $\alpha(\phi)$ defined above kills $v(z)$ by definition, and it is easy to check that it sends $u(z)$ to $v(z)\tensor \varphi(z)dz$.  Consider the oper $(E,L,\nabla+\alpha(\phi))$.  In the same  local basis of sections $(u(z),v(z))$, the  connection takes the form
\begin{equation}\label{opereq}\nabla+\alpha(\phi)=d+\mat{0}{1}{\varphi(z)}{0} dz.\end{equation}
This is exactly the rank two system associated to the second order equation \eqref{cat}. If we write a flat section in the form \[t(z)=y(z) u(z)-w(z)v(z)\] then we find that $w(z)=y'(z)$ with $y(z)$ satisfying \eqref{cat}. If we take two linearly independent solutions $y_i(z)$ of this equation we obtain a flat basis of sections $t_i(z)$ of the bundle $E$, in terms of which the line bundle $L$ is spanned by \[v(z)=\frac{1}{W(y_1,y_2)}\left(y_2(z) t_1(z)-y_1(z) t_2(z)\right),\] where the Wronskian $W(y_1,y_2)=y_1(z)y_2'(z)-y_1'(z)y_2(z)$ is a nonzero constant. Thus the projective chart defined by the new oper is given by the ratio $-y_1(z)/y_2(z)$, which is a ratio of solutions to~\eqref{cat} as required.
\end{proof}

\section{Meromorphic projective structures}
\label{sectmero}

In this section we introduce meromorphic projective structures and their associated marked bordered surfaces. We start by recalling some basic properties of meromorphic quadratic differentials on Riemann surfaces. For more on this the reader can consult \cite{BS}.

\subsection{Meromorphic quadratic differentials} \label{Sec:Flat}

 Let $S$ be a Riemann surface, and let  $\omega_S$ denote its holomorphic cotangent bundle. A \emph{meromorphic quadratic differential}  $\phi$ on $S$ is a meromorphic  section  of the line bundle $\omega_S^{\tensor 2}$.  Two such  differentials $\phi_1,\phi_2$ on  surfaces $S_1,S_2$  are said to be    equivalent if there is a biholomorphism  $f\colon S_1\to S_2$ such that $f^*(\phi_2)=\phi_1$.

In terms of a local co-ordinate $z$ on $S$ we can write a quadratic differential $\phi$ as
 \begin{equation}\label{form2}\phi(z)=\varphi(z)\, dz^{\tensor 2}\end{equation}
 with $\varphi(z)$ a  meromorphic function. 
 We write  $\Zer(\phi),\Pol(\phi)\subset S$ for the subsets of zeroes and poles of $\phi$ respectively. The subset  \[\Crit(\phi)=\Zer(\phi)\cup\Pol(\phi)\] is the set  of \emph{critical points}  of  $\phi$.

At any point of $S \setminus \Crit(\phi)$ there is a  distinguished local co-ordinate $w$, uniquely defined up to transformations of the form $w \mapsto \pm w + \constant$, with respect to which $\phi(w)=dw^{\tensor 2}$.
 In terms of an arbitrary local co-ordinate $z$ we have
\[w=\int \sqrt{\varphi(z)}\,dz.\]
A meromorphic quadratic differential $\phi$ on a Riemann surface $S$ determines two  foliations on $S\setminus \Crit(\phi)$,   the horizontal and vertical foliations. These are given in terms of a distinguished local co-ordinate $w$ by the lines  $\Im (w)=\constant$ and $\Re(w)=\constant$ respectively.

\subsection{Trajectory structure at poles}

We shall need some data naturally associated to poles of meromorphic quadratic differentials.
Suppose first that $p\in \Pol(\phi)$ is a  pole of order $\leq 2$ and let us write $\phi$ in the form \eqref{form2} for a local co-ordinate $z$ centered at $p$. Then  we define the \emph{leading coefficient} of $\phi$ at $p$ to be
\begin{equation}
\label{lee}a(p)=\lim_{z\to 0} z^2\cdot \varphi(z).\end{equation}
 We also define 
 the \emph{residue} of $\phi$ at $p$ to be 
\[
\res(p)=\pm 4\pi i \sqrt{a(p)},\]
which is well-defined up to sign. It is easily checked that these quantities are independent of the choice of co-ordinate $z$. Note that if $\phi$ has a simple pole at $p$ then the leading coefficient and residue are both defined to be zero.

\begin{figure}[ht]
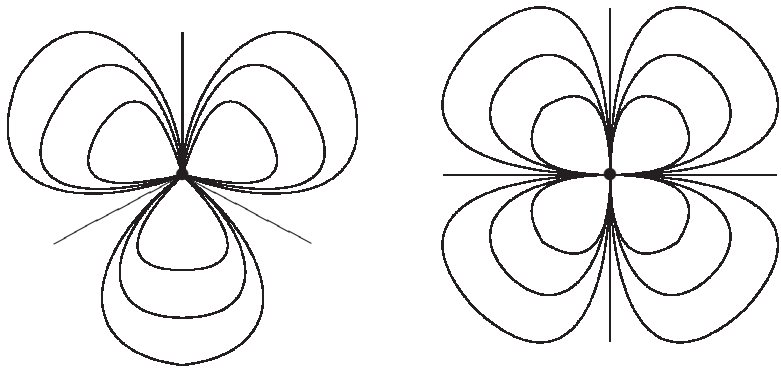

\begin{center}
\[
\xy /l3pc/:
{\xypolygon3"T"{~:{(2,0):}~>{}}},
{\xypolygon3"S"{~:{(1.5,0):}~>{}}},
{\xypolygon3"R"{~:{(1,0):}~>{}}},
(1,0)*{}="O"; 
(-0.35,0.72)*{}="U"; 
(2.35,0.72)*{}="V"; 
(1,-1.5)*{}="W"; 
"O";"U" **\dir{-}; 
"O";"V" **\dir{-}; 
"O";"W" **\dir{-}; 
"O";"T1" **\crv{(2,0.75) & (2.25,1.85)};
"O";"T1" **\crv{(0,0.75) & (-0.25,1.85)};
"O";"T2" **\crv{(0,0.4) & (-1.2,0.25)};
"O";"T2" **\crv{(1,-1) & (0,-2.2)};
"O";"T3" **\crv{(1,-1) & (2,-2.2)};
"O";"T3" **\crv{(2,0.4) & (3.2,0.25)};
"O";"S1" **\crv{(1.75,0.56) & (2,1.5)};
"O";"S1" **\crv{(0.25,0.56) & (0,1.5)};
"O";"S2" **\crv{(0,0.3) & (-0.9,0.19)};
"O";"S2" **\crv{(0.9,-0.8) & (0.3,-1.7)};
"O";"S3" **\crv{(2,0.3) & (2.9,0.19)};
"O";"S3" **\crv{(1.1,-0.8) & (1.7,-1.7)};
"O";"R1" **\crv{(1.5,0.5) & (1.75,1)};
"O";"R1" **\crv{(0.5,0.5) & (0.25,1)};
"O";"R2" **\crv{(0.5,0.1) & (-0.3,0.25)};
"O";"R2" **\crv{(0.75,-0.8) & (0.5,-1)};
"O";"R3" **\crv{(1.5,0.1) & (2.3,0.25)};
"O";"R3" **\crv{(1.25,-0.8) & (1.5,-1)};
(1,0)*{\bullet};
\endxy
\qquad
\xy /l3pc/:
{\xypolygon4"A"{~:{(2,0):}~>{}}},
{\xypolygon4"B"{~:{(1.5,0):}~>{}}},
{\xypolygon4"C"{~:{(1,0):}~>{}}},
(1,0)*{}="O"; 
(1,-1.75)*{}="T"; 
(-0.75,0)*{}="U"; 
(1,1.75)*{}="V"; 
(2.75,0)*{}="W"; 
"O";"T" **\dir{-}; 
"O";"U" **\dir{-}; 
"O";"V" **\dir{-}; 
"O";"W" **\dir{-}; 
"O";"A1" **\crv{(1,1.5) & (1.5,2.25)};
"O";"A1" **\crv{(2.5,0) & (3.25,0.5)};
"O";"A2" **\crv{(1,1.5) & (0.5,2.25)};
"O";"A2" **\crv{(-0.5,0) & (-1.25,0.5)};
"O";"A3" **\crv{(1,-1.5) & (0.5,-2.25)};
"O";"A3" **\crv{(-0.5,0) & (-1.25,-0.5)};
"O";"A4" **\crv{(1,-1.5) & (1.5,-2.25)};
"O";"A4" **\crv{(2.5,0) & (3.25,-0.5)};
"O";"B1" **\crv{(1,1) & (1.5,1.6)};
"O";"B1" **\crv{(2,0) & (2.6,0.5)};
"O";"B2" **\crv{(1,1) & (0.5,1.6)};
"O";"B2" **\crv{(0,0) & (-0.6,0.5)};
"O";"B3" **\crv{(1,-1) & (0.5,-1.6)};
"O";"B3" **\crv{(0,0) & (-0.6,-0.5)};
"O";"B4" **\crv{(1,-1) & (1.5,-1.6)};
"O";"B4" **\crv{(2,0) & (2.6,-0.5)};
"O";"C1" **\crv{(1,0.75) & (1.25,1)};
"O";"C1" **\crv{(1.75,0) & (2,0.25)};
"O";"C2" **\crv{(1,0.75) & (0.75,1)};
"O";"C2" **\crv{(0.25,0) & (0,0.25)};
"O";"C3" **\crv{(1,-0.75) & (0.75,-1)};
"O";"C3" **\crv{(0.25,0) & (0,-0.25)};
"O";"C4" **\crv{(1,-0.75) & (1.25,-1)};
"O";"C4" **\crv{(1.75,0) & (2,-0.25)};
(1,0)*{\bullet};
\endxy
\]
\end{center}
\caption{Local trajectory structures at poles of order $m=5$ and $6$\label{pole}.}
\end{figure}

Suppose now that $p\in \Pol(\phi)$ is a pole of order $m>2$ and take a local co-ordinate $z$ centered at $p$. Writing
\[\phi(z)=z^{-m}\cdot (a_0 + a_1 z + a_2 z^2 + \cdots) \cdot dz^{\tensor 2},\]
we define the \emph{asymptotic horizontal directions} of $\phi$ at $p$ to be the $m-2$ tangent rays defined by the condition $a_0 \cdot z^{2-m}\in \bR_{>0}$. Again, this is easily seen to be independent of the choice of the co-ordinate $z$. The reason for the name is that there is  a neighbourhood $p\in U\subset S$  such that any horizontal trajectory entering $U$ eventually tends to  $p$ and  becomes asymptotic to  one of the asymptotic horizontal directions: see Figure \ref{pole}. One can similarly define asymptotic vertical directions by the condition $a_0 \cdot z^{2-m}\in \bR_{<0}$.

\subsection{Meromorphic projective structures}

The affine space structure on the set of projective structures on a fixed Riemann surface allows  us to define a meromorphic projective structure as follows.

\begin{defn}\label{def} A \emph{meromorphic projective structure} $\cP$ on a Riemann surface $S$ is defined to be a projective structure $\cP^*$ on the complement $S^*= S\setminus P$ of a discrete subset $P\subset S$, such that given a holomorphic projective structure $\cP_0$ on $S$, the quadratic differential $\phi$ on $S^*$ defined as the difference  $\cP^*-\cP_0|_{S^*}$ extends to a meromorphic quadratic differential on $S$. This condition is independent of the choice of $\cP_0$.
\end{defn}

Let $\cP$ be a meromorphic projective structure on a Riemann surface $S$. Note that the meromorphic differential $\phi$ in Definition \ref{def} is uniquely determined   up to addition of holomorphic differentials. We call it a \emph{polar differential} for  $\cP$. We say that $\cP$ has a pole at a point $p\in S$ if  $\phi$ has a pole at that point, and we define the order of the pole to be the order of the pole of $\phi$. These notions are clearly independent of the choice of polar differential. We also refer to poles of order $\leq 2$ as \emph{regular singularities} of the projective structure, and to poles of order $>2$ as \emph{irregular singularities}.

\begin{lemma}
\label{triv}
The set of meromorphic projective structures on a fixed Riemann surface $S$ is an affine space for the vector space of meromorphic quadratic differentials on $S$.
\end{lemma}

\begin{proof}
First consider the following  general construction. Suppose $\bA$ is an affine space for a vector space $V$, and suppose $V\subset W$ is a subspace. Then we can define an affine space for $W$ by setting
\[\bA\oplus_V W:=(\bA\times W)/V,\]
where $v\in V$ acts on $(p,w)\in \bA\times W$ by mapping it to $(p+v,w-v)$. Applying this construction  to the inclusion of holomorphic quadratic differentials inside meromorphic ones, and the affine space of holomorphic projective structures, it is easy to see that the points of the resulting affine space are in bijection with the set of meromorphic projective structures defined above.
\end{proof}

\subsection{Meromorphic opers}
\label{meroper}
Let $D$ be an effective divisor on a Riemann surface $S$. Recall that a meromorphic connection  on a bundle $E$ with polar divisor $D$ is a $\bC$-linear map of sheaves
\[\nabla\colon E\to E\tensor \omega_S(D)\]
which satisfies the Leibniz rule: $\nabla(fs)=s\tensor df + \nabla(s)\tensor f$,  where we interpret $df$ as a section of $\omega_S(D)$ via the canonical inclusion $\omega_S\to \omega_S(D)$. Of course, a holomorphic connection on $E$ induces a meromorphic one by composing with the canonical inclusion $E\tensor\omega_S\to E\tensor \omega_S(D)$.

We will not develop a detailed theory of meromorphic opers here. However it will be important for us to know that a meromorphic projective structure $\cP$ has an associated meromorphic connection $(E,\nabla)$. To construct it, let us write $\cP=\cP_0+ \phi$, with $\cP_0$ a holomorphic projective structure, and
\[\phi\in H^0(S,\omega_S^{\tensor 2}(D)),\]
a quadratic differential with polar divisor $D$. Then, as explained in Section \ref{plob}, the holomorphic projective structure  $\cP_0$ has an associated $\GL(2,\bC)$ oper $(E,L,\nabla_0)$, well-defined up to  tensoring with  line bundles with flat connection $(M,\nabla_M)$ as in Section \ref{pot}. Exactly as in Section \ref{ao}, we can form the linear map\[\alpha(\phi): E\lRa{q} E/L \lRa{\theta^{-1}\tensor \id_{\omega_S^{-1}}} L\tensor \omega_S^{-1}\lRa{\id_{L\tensor \omega_S^{-1}}\tensor \phi}  L\tensor \omega_S(D) \lRa{i\tensor \id_{\omega_S(D)}} E\tensor \omega_S(D),\]
and hence define a flat meromorphic connection $\nabla=\nabla_0+\alpha(\phi)$ on $E$. It is easy to check using Lemma \ref{plob} that the resulting connection is well-defined up to tensoring with  line bundles with flat connection $(M,\nabla_M)$ as in Section \ref{pot}, and is independent of the choice of base holomorphic projective structure $\cP_0$.

\begin{remarks}
\label{snobl}
\begin{itemize}\item[(i)]
Note that 
 the  $\O_S$-linear map defined as the composite\[
L\lRa{i} E \lRa{\nabla} E\tensor \omega_S(D) \lRa{q\tensor \id_{\omega_S(D)}} (E/L)\tensor \omega_S(D)\] is the map $\theta$ of \eqref{theta} for the oper $(E,L,\nabla_0)$, composed with the embedding \[(E/L)\tensor \omega_S\to (E/L)\tensor \omega_S(D).\] It is therefore an isomorphism away from the poles of $\cP$.
\item[(ii)]
The same argument as Lemma \ref{plob} shows that in a local basis of flat sections of the connection $\nabla_0$ the connection $\nabla$ is given  by
\[
\nabla=d+\mat{0}{1}{\varphi(z)}{0}dz.\]
Of course, this coincides with the rank two differential system associated to the second order equation \eqref{cat}. \end{itemize}
\end{remarks}

\subsection{Regular singularities}
\label{sig}

Suppose that $S$ is a Riemann surface and $\cP$ is a meromorphic projective structure with a regular singularity at a point $p\in S$. We  define the \emph{leading coefficient} $a(p)$ of the projective structure $\cP$ at the point $p$ to be the leading coefficient \eqref{lee} of a corresponding polar  differential.  We also introduce a quantity 
\begin{equation}
\label{resnot}r(p)=\pm 2\pi i \sqrt{1+4a(p)},\end{equation}
well-defined up to sign, which we call the \emph{exponent} of $\cP$  at the regular singularity $p$. 

The meromorphic projective structure $\cP$ restricts to a holomorphic projective structure on a punctured neighbourhood of $p\in S$ and the corresponding monodromy local system determines a well-defined conjugacy class of elements in $G$ which we call the \emph{monodromy around the regular singularity $p$}. We will see in Section \ref{locale} that this monodromy  is controlled by the exponent $r(p)$.

\begin{example}
\label{egg}
There is a holomorphic projective structure $\cP$ on $\bC^*$ whose charts are given by branches of $\log(z)$. To compare it with the standard projective structure $\cP_0$ on $\bC$ given by the single chart $z\colon \bC\to \bC$ we must compute the Schwarzian derivative \eqref{schwarz} with $f(z)=\log(z)$. The result is that $\cP-\cP_0$ is the quadratic differential
\[\phi(z)=-\frac{dz^{\tensor 2}}{4z^2}.\]
We therefore conclude that  $\cP$ is a meromorphic projective structure on $\bC$ with a double pole at the origin, with leading coefficient $a(p)=-\frac{1}{4}$ and exponent $r(p)=0$.
\end{example}

We say that the point $p\in S$ is an \emph{apparent singularity} of the projective structure $\cP$ if it is a regular singularity whose monodromy is the identity element of $G$. We define a \emph{signing} of a meromorphic projective structure $\cP$ to be a choice of sign of the exponent $r(p)$ at each  regular singularity  $p\in S$. A \emph{signed meromorphic projective structure} is a projective structure equipped with a signing.

\subsection{Marked bordered surfaces}
\label{mbs}

A \emph{marked bordered surface} is defined to be a pair $(\bS,\bM)$  consisting of a compact, connected, oriented, smooth surface with boundary, and a finite non-empty set $\bM\subset \bS$ of marked points, such that each boundary component of $\bS$ contains at least one marked point. Marked points in the interior of $\bS$ are called \emph{punctures}. The set of punctures is denoted $\bP\subset \bM$, and we write $\bS^*=\bS\setminus \bP$ for the corresponding punctured surface.

An isomorphism of marked bordered surfaces $(\bS_i,\bM_i)$ is an orientation-preserving diffeomorphism $f\colon \bS_1\to \bS_2$ between the underlying surfaces which induces a bijection $f\colon \bM_1\to \bM_2$. Two such isomorphisms are called isotopic if the underlying diffeomorphisms are related by an isotopy through diffeomorphisms $f_t\colon \bS_1\to \bS_2$ which also induce bijections $f_t\colon \bM_1\to \bM_2$.

It is sometimes more natural to replace the surface $\bS$ with the surface $\bS'$ obtained by taking the real oriented blow-up of $\bS$ at each puncture: this replaces  punctures  with boundary circles containing no marked points. 
A marked bordered surface $(\bS,\bM)$ is determined up to isomorphism by its genus $g=g(\bS)$ and an unordered collection of non-negative integers $\{k_1,\cdots,k_d\}$ giving the number of marked points on the boundary circles of the surface $\bS'$. An associated integer which will come up very often in what follows is
\begin{equation}\label{n}n = 6g-6+\sum_i (k_i+3).\end{equation}

The mapping class group $\MCG(\bS,\bM)$ of a marked bordered surface is defined to be the group of all isotopy classes of isomorphisms from $(\bS,\bM)$ to itself.

\begin{example}
Consider the surface $(\bS,\bM)$ consisting of a disc $\bS$ with $n$ marked points $\bM\subset \partial \bS$ lying on the boundary. An isomorphism from $(\bS,\bM)$ to itself is an orientation-preserving diffeomorphism $f\colon \bS\to \bS$ satisfying $f(\bM)\subset \bM$. The set $\bM$ has a natural cyclic ordering, and $f$ must induce a cyclic permutation. In this way one can see that $\MCG(\bS,\bM)\isom \bZ/n\bZ$ is a cyclic group of order $n$.
\end{example}

\subsection{Labelling meromorphic projective structures}
\label{la}
Suppose that $\phi$ is a meromorphic quadratic differential on a compact Riemann surface $S$ with at least one pole. We define an associated  marked bordered surface $(\bS,\bM)$ by the following construction.   To define the surface $\bS$ we take the  underlying smooth surface of $S$ and perform  an oriented real blow-up  at each pole of $\phi$ of order~$ >2$.  The marked points $\bM$ are defined to be the poles of $\phi$ of order $\leq 2$, considered as points of the interior of $\bS$, together with the points on the boundary of $\bS$ corresponding to the asymptotic horizontal directions of $\phi$.

Let $\cP$ be a meromorphic projective structure on a compact Riemann surface $S$ with at least one pole. We  define the marked bordered surface associated to the pair $(S,\cP)$  to be the marked bordered surface  associated to a polar differential of $\cP$. Note that this is well-defined because two meromorphic differentials $\phi_i$ on the same Riemann surface $S$ give rise to identical marked bordered surfaces if the difference $\phi_2-\phi_1$ is holomorphic. Indeed, the two differentials clearly have the same poles and the same asymptotic horizontal directions.

More concretely, the marked bordered surface $(\bS,\bM)$ associated to a meromorphic projective structure $\cP$ on a compact Riemann surface $S$ is determined as above by its genus $g=g(\bS)$ and the non-empty collection of non-negative integers $\{k_1,\cdots,k_d\}$. These can be read off from the pair $(S,
\cP)$ as follows. The genus is $g=g(S)$ and the number $d\geq 1$ is the number of poles of the projective structure $\cP$. Each  pole of $\cP$ of order $m_i\geq 3$ corresponds to an integer $k_i=m_i-2$, whereas each pole of order $m_i\leq 2$ corresponds to an integer $k_i=0$.

\section{Framed local systems}
\label{fr}

In this section we recall from \cite{IHES} the definition of the stack of framed $G=\PGL_2(\bC)$ local systems on a marked bordered surface. We also introduce the open substack of non-degenerate framed local systems.

\subsection{Framed local systems}
\label{framedlocalsystems}

Let $(\mathbb{S},\mathbb{M})$ be a marked bordered surface.  We shall be interested in $G$ local systems $\cG$ on the associated punctured surface $\bS^*$. By definition, $\cG$ is a principal $G$ bundle over $\bS^* $ equipped with a flat connection $\nabla$. The group $G$ acts naturally on $\bP^1$ on the left, so we can form the associated $\bP^1$ bundle\begin{equation}\label{l}
\cL=\cG\times_{G}\bP^1,
\end{equation}
 which inherits a flat connection from $\nabla$. 
For each point $p\in \bM$ let us fix a small contractible open neighbourhood  $p\in U(p)\subset \bS$.

A \emph{framed $G$ local system} on $(\bS,\bM)$ is a  pair $(\cG, \ell(p))$ consisting of  a $G$ local system $\cG$ on the surface $\bS^*$ equipped with the data of a framing: a choice of flat section $\ell(p)$ of the associated bundle $\mathcal{L}$ over each of the subsets  $V(p)=U(p)\cap \bS^*$. 
An isomorphism between two  framed $G$ local systems $(\cG_i,\ell_i(p))$ on $(\bS,\bM)$ is an isomorphism $\theta\colon \cG_1\to \cG_2$  between the underlying $G$ local systems on $\bS^*$, which preserves the framings, in the sense that $\theta(\ell_1(p))=\ell_2(p)$ for all $p\in \bM$.

\begin{remark}\label{so}If $p\in \bP$ is a puncture, the subset  $V(p)$ can be taken to be a small punctured disc in the interior of $\bS$, and the corresponding flat section $\ell(p)$ is necessarily invariant under parallel transport  around the loop  encircling $p$.\end{remark}

Let us fix a base-point $x\in \bS^*$. We also choose, for each point $p\in \bM$, a path $\beta_p$ connecting $x$ to~$p$ whose interior lies in $\bS^*$. Then, for each puncture $p\in \bP$, we can define an element  $\delta_p\in \pi_1(\bS^*,x)$  by travelling from $x$ to $V(p)$ along the path $\beta_p$, encircling $p$ counter-clockwise by a small loop, and then returning to $x$ along $\beta_p$.

By a \emph{rigidified framed $G$ local  system}  we mean a framed $G$ local system $(\cG, \ell(p))$ together with the choice  of  an element $s$ of the fibre $\cG_{x}$. An isomorphism between two rigidified framed $G$ local systems $(\cG_i, \ell_i(p),s_i)$ is an isomorphism $\theta$ between  the underlying framed $G$ local systems as above, which satisfies  $\theta(s_1)=s_2$.

\begin{lemma}
\label{sore}
There is a bijection between the set of isomorphism classes of  rigidified  framed $G$ local systems on $(\bS,\bM)$ and the set of points of the  complex projective variety
\[X(\bS,\bM)=\Big\{\rho\in \Hom_{\Grp}(\pi_1(\bS^*,x),G), \lambda\in \Hom_{\Set}(\bM,\bP^1):  \rho(\delta_p)(\lambda(p))=\lambda(p)\text{ for all }p\in \bP\Big\}.\]
\end{lemma}

\begin{proof}
Consider a rigidified framed $G$-local system $(\cG,\ell(p),s)$. The element $s\in \cG_x$  gives an identification $G\isom \cG_x$, defined by mapping $g\mapsto s\cdot g$. This induces an identification  $\cL_x\isom \bP^1$. Parallel transport along $\cG$ defines a group homomorphism $\rho\colon \pi_1(\bS^*,x)\to G$. Parallel transport of the flat section $\ell(p)$  along the path $\beta_p$ gives an element $\lambda(p)\in \cL_x\isom \bP^1$. Remark \ref{so} implies that there is a  relation $\rho(\delta_p)(\lambda(p))=\lambda(p)$ for each puncture $p\in \bP$. 

This construction defines a map from the set of isomorphism classes of  rigidified  framed $G$ local systems on $(\bS,\bM)$ to the set of points of $X(\bS,\bM)$. The proof that this is a bijection is left to the reader; the only non-trivial point is the well-known fact that isomorphism classes of rigidified $G$ local systems on $\bS^*$ correspond bijectively to elements of $\Hom_{\Grp}(\pi_1(\bS^*,x),G)$.
\end{proof}

The group $G$ acts on the set of isomorphism classes of rigidified framed local systems $(\cG,\ell(p),s)$ by fixing the underlying framed local system $(\cG,\ell(p))$ and mapping  $s\mapsto s\cdot g$. The induced action on $X(\bS,\bM)$ is given by
\[g\cdot (\rho,\lambda) =(g\cdot \rho\cdot  g^{-1}, g\circ \lambda).\]
The moduli stack of framed $G$-local systems is by definition the quotient stack
\[\cX(\bS,\bM)=X(\bS,\bM)/G.\]
Although we will occasionally find it convenient to use  the language of stacks in what follows, there is in fact no real need for this, since we will  be mainly interested in the  open subvariety of non-degenerate rigidified framed local systems $X^*(\bS,\bM)$ introduced in the next section, for which the corresponding quotient $\cX^*(\bS,\bM)$ is a (possibly non-Hausdorff) complex manifold.

\subsection{Degenerate framed local systems}
\label{bob}

We now introduce an open subset of the stack of  framed local systems whose points we call {non-degenerate framed local systems}. This subset will be small enough that it forms a complex manifold, but large enough that it contains the image of our generalised monodromy map.  For the purposes of this definition, we recall that an element of $G=\PGL_2(\bC)$ is called \emph{parabolic} if it  has a single fixed point on $\bP^1$. 

\begin{defn}
\label{rn}
A framed local system $(\cG,\ell(p))$ on a marked bordered surface $(\bS,\bM)$ is said to be \emph{degenerate} if one of the following  three conditions holds:
\begin{itemize}
\item[(D1)] There is a connected component $I$ of the punctured boundary $\partial \bS\setminus \bM$ such that the two chosen sections $\ell(p_i)$ defined near the points $p_1,p_2\in \bM$ lying at the ends of  $I$    coincide under parallel transport along $I$.
\smallskip

\item[(D2)]There exists an unordered  pair of distinct points  $\ell_\pm(p)$ in each fibre of the bundle $\cL$, each of which locally defines a flat section of $\cL$ but which may be permuted by the monodromy of $\cL$,  such that for each $p\in \bM$ the chosen section $\ell(p)$ coincides with one of the $\ell_\pm(p)$.
  \smallskip

\item[(D3)]  The surface $\bS$ is closed, there  exists a single flat section $\ell$ of the bundle $\cL$ over $\bS^*$ such that for each $p\in \bM$ the chosen section $\ell(p)$ coincides with $\ell$,   and moreover, the monodromy around each puncture is either parabolic or the identity.
 \end{itemize}
\end{defn}
 
   If none of these conditions holds we call  the framed local system $(\cG,\ell(p))$ \emph{non-degenerate}.   We consider a rigidified framed $G$ local system to be degenerate  or non-degenerate precisely if the underlying framed $G$ local system is.

\begin{remarks}
\label{remas}
\begin{itemize}
\item[(i)] Note that for any given surface $(\bS,\bM)$ we need only ever consider two of the above conditions, since when $\bS$ is  closed  condition (D1) is vacuous. Condition (D1) relates to the standard opposedness property of subdominant solutions: see Proposition~\ref{e}(i) below.\smallskip

\item[(ii)] Condition (D2) is perhaps more easily stated in terms of rigidified framed local systems, and the corresponding points of the variety $X(\bS,\bM)$. It follows easily from the definition that a point $(\rho,\lambda)$ of $X(\bS,\bM)$ corresponds under the bijection of Lemma \ref{sore} to a rigidified framed local system satisfying condition (D2)  precisely if there is an unordered  pair of points $\{a,b\}\in \bP^1$  which are fixed or permuted by all monodromies $\rho(\gamma)$, and which include all points $\lambda(p)$ for $p\in \bM$.\smallskip

\item[(iii)] Note that when a framed local system satisfies condition (D2), the two locally-defined sections $\ell_\pm$  are fixed individually, rather than permuted, by the monodromy around any puncture. Indeed,  one of these two points is the framing $\ell(p)$, which is fixed by definition. In particular, if the monodromy around a puncture is not the identity then the fixed points of its action on $\cL$  are precisely the lines $\ell_\pm$.\smallskip

\item[(iv)] A framed local system satisfying condition (D2) gives rise to a permutation representation
\begin{equation}
\label{lablab}\sigma\colon \pi_1(\bS)\to \{\pm 1\},\end{equation}
such that $\rho(\gamma)(\ell_+)=\ell_{\sigma(\gamma)}$.  This  in turn defines an unramified double cover $\pi\colon {\bT}\to \bS$. Pulling back the framed local system to this double cover gives a framed local system on~$\bT$ satisfying condition (D2) in the stronger sense that the associated $\bP^1$ bundle has two flat sections $\ell_\pm$ globally defined over the punctured surface $\bT^*$ which restrict to give all framings.\smallskip

\item[(v)] Let $(\cG,\ell)$ be a framed local system and suppose that the  monodromy around some puncture $p\in \bP$ is neither parabolic nor the identity. Then this monodromy has two eigenlines, and we can replace the framing $\ell(p)$ at the puncture $p$ by the  other eigenspace to give a new framed local system. We claim that this flipped framed local system is degenerate precisely if the original one is. Indeed, this statement is immediate for conditions (D1) and (D3),  and the claim for (D2) follows from easily from (iii) above.
\end{itemize}
\end{remarks}

\subsection{Basic properties of non-degeneracy}
The following result shows that the notion of non-degeneracy defined above is an open condition on rigidified framed local systems.

\begin{lemma}
Under the bijection of Lemma \ref{sore}, the subset of isomorphism classes of  non-degenerate  rigidified  framed $G$-local systems   corresponds to a Zariski open subset \[X^*(\bS,\bM)\subset X(\bS,\bM).\]
\end{lemma}

\begin{proof}
For each $n=1,2,3$ we denote by $X_n\subset X=X(\bS,\bM)$  the subset of rigidified framed local systems satisfying the corresponding condition (D$n$) of Definition \ref{rn}. We must show that the union  $X_1\cup X_2\cup X_3$ is a closed subset. 

We first claim that the subset $X_1$ is closed. Indeed, consider a connected component  $I\subset \partial \bS\setminus \bM$, with endpoints $p,q\in \bM$. Let $\delta_{I}\in \pi_1(S,x)$ be the element obtained by following the path $\beta_p$  from $x$ to $p$, traversing the boundary arc $I$ from $p$ to $q$, and then returning to $x$ along the path $\beta_q$.   A point $(\rho,\lambda)$ lies in $X_1$ precisely if for some such $I$ one has  $\rho(\delta_{I})(\lambda(p))= \lambda(q)$. This is clearly a closed condition.

The subset $X_3$ is also closed because, when $\bS$ is closed,  it coincides with the set of points $(\rho,\lambda)$  satisfying
\begin{itemize}
\item[(i)]
all points of $\bP^1$ of the form $\rho(\gamma)(\lambda(p))$ are equal,
\item[(ii)]  all  elements $\rho(\delta_p)$ are parabolic or the identity.
\end{itemize}
The second of these conditions is indeed closed because an element $M\in \GL(2,\bC)$ defines a parabolic element of $G$ or the identity precisely if $\tr(M)^2=4\det(M)$.

To complete the proof we will prove that the closure of $X_2$ is contained in  $X_1\cup X_2\cup X_3$.   Consider a sequence of points $(\rho_r,\lambda_r)\in X_2$ converging as $r\to \infty$ to some point $(\rho,\lambda)\in X(\bS,\bM)$. For each~$r$ there is a corresponding unordered pair of points $\{a_r,b_r\}\in \bP^1$ as in Remark \ref{remas} (ii). We can choose an arbitrary ordering of these,  and passing  to a subsequence assume that the sequences $a_r$ and $b_r$ converge to (not necessarily distinct) points $a,b\in \bP^1$. By continuity,  the points $a,b\in \bP^1$ are fixed or permuted by all monodromies $\rho(\gamma)$, and include all points $\lambda(p)$ for $p\in \bM$. If the points $a$ and $b$ are distinct we conclude that $(\rho,\lambda)\in X_2$. On the other hand, if  $a=b$, and $\bS$ has non-empty boundary, then $(\rho,\lambda)\in X_1$.
 
To conclude the proof, we claim that if $a=b$ and $\bS$ is closed, then all puncture monodromies $\rho(\delta_p)$  are either parabolic or the identity, so that $(\rho,\lambda)\in X_3$. To see this, note that  $\rho(\delta_p)$ is the limit as $r\to \infty$ of elements of $\rho_r(\delta_p)\in G$ which are either the identity, or whose fixed points set consists of the two points $\{a_r,b_r\}$. Passing to a subsequence we can assume that only one of these two possibilities occurs. It is then easy to see that 
$\rho(\delta_p)$ is either parabolic or the identity.
\end{proof}

We have not been able to show directly that the quotient
\[\cX^*(\bS,\bM)=X^*(\bS,\bM)/G\]
is a  complex manifold, although we will prove this indirectly in Section \ref{final}. It is quite straightforward however to show that this action is free, so we include a proof of this.

\begin{lemma}
The action of $G$ on the variety $X(\bS,\bM)$ considered above restricts to a free action on the open subset $X^*(\bS,\bM)$.
\end{lemma}

\begin{proof}
Suppose  that some non-trivial element $g\in G$ fixes an element $(\rho,\lambda)$ of $X(\bS,\bM)$. We must show that the rigidified framed local system corresponding to $(\rho,\lambda)$ is degenerate. 

Suppose  first that $g$ has two fixed points $\{a,b\}\in \bP^1$.  Recall that two non-trivial elements of $G$ commute precisely if they have the same set of fixed points in $\bP^1$. Since the elements $\rho(\gamma)\in G$ all commute with $g$, they must also preserve the points $\{a,b\}$. Moreover since $g$ fixes all the points $\lambda(p)\in \bP^1$, the image of $\lambda$ must be contained in $\{a,b\}$. This is precisely the condition that  $(\rho,\lambda)$ satisfies the condition (D2) of Definition \ref{rn}.

The other possibility is that $g$ has a single fixed point. Then, by the same argument, the map $\lambda$ is constant, and all non-trivial elements $\rho(\gamma)\in G$ have a single fixed point. If $\bS$ has non-empty boundary, it follows that  $(\rho,\lambda)$ satisfies condition (D1) of Definition \ref{rn}, while if $\bS$ is closed, it satisfies condition (D3). \end{proof}

\section{Local study of singularities}

In this section we study the local properties of a meromorphic projective structure in a neighbourhood of a pole. This material is quite well-known but for the benefit of the reader we include simple proofs of most of the statements. 

 \subsection{Local description}
  
\label{localstudy}
Consider a quadratic differential \begin{equation}
\label{langset}\phi(z)=\varphi(z)\, dz^{\tensor 2},\end{equation} which is  meromorphic in a disc $\Delta\subset \bC$ centered at the origin  $z=0$ in the complex plane, and which has a pole of order $m\geq 1$ at the origin. Shrinking the disc if necessary we can assume that $\varphi(z)$ has no other poles in $\Delta$. We can write
\[\phi(z)=z^{-m}\cdot (a_0 + a_1 z + a_2 z^2 + \cdots) \cdot dz^{\tensor 2}.\]

Since $\Delta\subset \bC$ has a standard projective structure, the differential $\phi(z)$  induces a meromorphic projective structure $\cP$ on the Riemann surface $\Delta$ whose charts over open subsets of $\Delta^*=\Delta\setminus\{0\}$ are given by quotients of linearly independent solutions to the equation
\begin{equation}
\label{schrod3}
y''(z)-\varphi(z)\cdot y(z)=0,\end{equation}
which is equivalent to the linear system
\begin{equation}
\label{bl} \frac{d}{dz} \vect{y(z)}{-y'(z)} +A(z) \vect{y(z)}{-y'(z)}=0, \qquad A(z)=\mat{0}{1}{\varphi(z)}{0} .\end{equation}

We will focus mostly on the case when $m=2q$ is even, since we can reduce the general case to this one by taking a branched covering. Gauge transforming by the matrix $Q(z)=\diag(z^{-q},1)$  then has the effect of replacing $A(z)$ by
\begin{equation}
\label{q}B(z)=Q(z)A(z)Q^{-1}(z) - Q'(z)Q^{-1}(z)=\mat{qz^{-1}}{z^{-q}}{z^q\cdot \varphi(z)} {0}.\end{equation}
The leading order term of $B(z)$  as $z\to 0$ is then the $z^{-q}$ term, which is easily seen to be diagonalisable. 
\subsection{The regular case}

Let us consider the situation of Section \ref{localstudy} in the case $m\leq 2$. Note that $a=a_0$ is the leading coefficient of $\cP$ at the pole $z=0$ as defined in Section \ref{sig}. Recall also the definition of the exponent
\[r=\pm 2\pi i \cdot \sqrt{1+4a}.\]
There is a well-defined conjugacy class of elements in $G$ given by the monodromy of the projective structure $\cP$ around $z=0$.  

\begin{lemma}
\label{above}
The eigenvalues of the monodromy of the system \eqref{bl} around $z=0$ are
 \[\lambda_\pm=-\exp(\pm r/2).\]
 \end{lemma}

\begin{proof}
Consider first the case $m=2$ so that the leading term $a$ is nonzero. The system determined by \eqref{q} has a logarithmic singularity at $z=0$, and the residue matrix $B_0$ has  determinant $-a$ and trace $1$.
The eigenvalues of $B_0$  are therefore
\[ \lambda_\pm (B_0)=\frac{1}{2}\big (1\pm \sqrt{1+4a}\big )= \frac{1}{2}\big(1 \pm \frac{1}{2\pi i} r\big).\]
Since it is a general fact that the monodromy  has the same characteristic polynomial as $\exp(2\pi i B_0)$ (see e.g. \cite[Exercise II.4.5]{Sabbah}) the result follows.
Note that the  result also holds when  $a=0$. Indeed, in that case the function $\varphi(z)$ has a simple pole at $z=0$, and  the original system \eqref{bl} has a logarithmic singularity, whose residue matrix is nilpotent.
\end{proof}

We now have several possibilities:
\begin{itemize}
\item[(a)] If the exponent $r\notin (2\pi i)\cdot \bZ$ then the  residue matrix $B_0$ has distinct eigenvalues which do not differ by an integer. The monodromy   is therefore diagonalisable with distinct eigenvalues $-\exp(\pm r/2)$. There is a distinguished eigenline corresponding to each choice of sign  of $r$. \smallskip

\item[(b)] If $r=2\pi i n$ with $n\in \bZ$ the residue matrix $B_0$ has eigenvalues which differ by an integer. The monodromy  has repeated eigenvalue $(-1)^{n+1}$. There are then two cases:
\begin{itemize}
\item[(i)] The monodromy  is $(-1)^{n+1}\cdot I$. In this case the projective structure $\cP$ has an apparent singularity at $z=0$.

\item[(ii)] The monodromy  is conjugate to $(-1)^{n+1}\cdot U$, with $U$ a non-trivial unipotent matrix. The monodromy is then a parabolic element of $G$.
\end{itemize}
\end{itemize}

\subsection{The irregular case}
\label{local}

Let us consider again the situation of Section \ref{localstudy} but now assume that $m>2$. Once again we set $a=a_0$. We define the \emph{anti-Stokes rays} of the equation \eqref{schrod3}  at $z=0$ to be the asymptotic vertical directions of the quadratic differential $\phi(z)$. They are the $m-2$ directions where the expression $a\cdot z^{2-m}$ is real and negative. We refer to the closed sectors bounded by them as the \emph{Stokes sectors}.

\begin{thm}
\label{d}
In the interior of each Stokes sector there is a unique up-to-scale  vector-valued solution $v(z)$ to the system \eqref{bl} with $v(z)\to 0$ as $z\to 0$. This immediately gives a unique up-to-scale solution $y(z)$ to the equation \eqref{schrod3} such that $y(z)\to 0$ as $z\to 0$. 
 \end{thm}

\begin{proof}
Consider first the case when $m=2q$ is even. The leading order term of the matrix \eqref{q} has determinant $-a$ and trace $0$, so after a constant gauge change the corresponding matrix differential equation becomes equivalent to
\begin{equation}\label{may}Y'(z) = M(z)\cdot Y(z), \quad M(z)=\sum_{i=0}^\infty M_i \cdot z^{i-q}, \quad M_0=\mat{\sqrt{a}}{0}{0}{-\sqrt{a}}.\end{equation}
 This agrees with the form considered by Balser, Jurkat and Lutz \cite{BJL} (their variable $z$ is our $z^{-1}$, and their positive integer $r$ is our $q-1$). 
There is a unique formal solution of the form
\[
\Hat{Y}(z)=F(z) \cdot \exp Q(z),
\]
where $F(z)=\sum_{i=0}^\infty F_i \cdot z^i$ is a formal power series with $F_0=1$ the identity matrix,  and \[Q(z)=Q_0\cdot \log(z)+\sum_{i=1}^{q-1} Q_i\cdot z^{-i}\]
is a finite sum, with each matrix $Q_i$ diagonal, and $Q_{q-1}= -M_0/(q-1)$. 

Note that the behaviour of the factor $\exp Q(z)$ as $z\to 0$ is  dominated by the diagonal matrix \[\exp(-M_0\cdot z^{1-q}/(q-1)),\] whose entries grow or decay exponentially as $z\to 0$ along a ray depending on the sign $\epsilon(z)\in\{\pm 1\}$ of the real part of the expression $\sqrt{a} \cdot z^{1-q}$.  Note that this sign $\epsilon(z)$ is constant in each Stokes sector and flips when an anti-Stokes ray is crossed.

Consider an open sector $\Sigma$ of the $z$-plane with angle $\pi/(q-1)$. The basic existence result Theorem A of \cite{BJL} shows that  there is a distinguished fundamental solution
\begin{equation}
\label{dol}Y\colon \Sigma\to \GL_2(\bC)\end{equation}
 to \eqref{may} in this sector with the property that
\begin{equation}
\label{fl2}Y(z)\cdot \exp (-Q(z))\to 1\qquad \text{as }z\to 0 \text{ in }\Sigma.\end{equation}
 This is  stated in \cite{BJL} only for sectors bounded by the rays spanned by the $2(q-1)$ roots of unity, but the general case follows immediately by rescaling $z$ in the initial equation. 
 It follows that the columns of the matrix $Y(z)$ give vector-valued solutions to \eqref{may} which have either exponential growth or decay as $z\to 0$ along rays in $\Sigma$ depending on the sign $\epsilon(z)$. Tracing back the meromorphic gauge change to \eqref{bl} this then gives the claim.

Let us now suppose that $m$ is odd. Consider the branched double cover
\[f\colon \Delta\to \Delta, \qquad w\mapsto z=w^2.\]
Pulling back the standard projective structure on $\Delta$ gives a projective structure with a pole of order 2 at $w=0$, as is easily checked by computing the Schwarzian derivative
 \[-\frac{1}{2}\bigg[ \bigg(\frac{(w^2)''}{(w^2)'}\bigg)'- \bigg(\frac{(w^2)''}{(w^2)'}\bigg)^2\bigg]=\frac{1}{w^2}.\]
Pulling back the  quadratic differential \eqref{langset} gives the differential \[f^*(\phi)(w)=4w^2 \varphi(w^2)\, dw^{\tensor 2},\]
which has a pole of order $2m-2>2$ at the point $w=0$.  Thus  the meromorphic projective structure $\cP$ defined by ratios of solutions to \eqref{schrod3} pulls back to give a meromorphic projective structure with a pole of even order $2m-2$. We can then apply the  argument above to obtain unique subdominant solutions in the sectors bounded by the asymptotic horizontal directions of $f^*(\phi)$ and then project back down via the double cover $f$  to give the result.
 \end{proof}

The solutions to \eqref{schrod3} or equivalently \eqref{bl} given by Theorem \ref{d} are called  \emph{subdominant} solutions. In terms of the marked bordered surface $(\bS,\bM)$ associated to the quadratic differential $\phi(z)$, they give  distinguished solutions in a neighbourhood of each marked point on the boundary of $\bS$.

\subsection{More on Stokes data}

Consider again the  situation of Section \ref{localstudy} and assume again that $m>2$. We shall need the following properties of the subdominant solutions defined in the last section. 

\begin{prop}
\label{e}

\begin{itemize}
\item[(i)]  Suppose two Stokes sectors $\Sigma_1,\Sigma_2$ meet along an anti-Stokes ray $\ell$, and for each $i$ choose a vector-valued solution $v_i(z)$ to the system  \eqref{bl} which is subdominant in the sector $\Sigma_i$. Then the $v_i(z)$ analytically continue  across $\ell$ to give a basis of solutions   in  a neighbourhood of $\ell$.
\smallskip

\item[(ii)]Suppose given a  pair of vector-valued solutions $v_1(z),v_2(z)$ to the system \eqref{bl} defined in the punctured disc $\Delta^*$, such that in each Stokes sector $\Sigma$  one of the two solutions $v_1(z),v_2(z)$ is subdominant. Then there exist holomorphic functions $\alpha_i\colon \Delta^*\to \bC^*$ such that the vector-valued functions $\alpha_i(z)\cdot v_i(z)$  extend over the origin and are linearly-independent for all $z\in \Delta$. \end{itemize}
\end{prop}

\begin{proof}
Consider an open sector $\Sigma$ of the $z$-plane of angle $\pi/(q-1)$ and a corresponding distinguished fundamental solution \eqref{dol}. The uniqueness argument \cite[Remark 1.4]{BJL} shows that if the sector $\Sigma$ is not a Stokes sector, then $Y(z)=Y_\Sigma(z)$ is unique with the property \eqref{fl2}. Indeed, any two such solutions would differ by a constant matrix $C$ satisfying
\[\exp Q(z)\cdot C\cdot \exp (-Q(z))\to 1 \text{ as } z\to 0 \text{ in }\Sigma.\]
But if $\Sigma$ contains an anti-Stokes ray then each entry of the diagonal matrix $\exp Q(z)$ exhibits both exponential growth and decay along different rays in $\Sigma$, which forces $C$ to be the identity.

To prove (i) consider an open sector $\Sigma_{12}$ of angle $\pi/(q-1)$ overlapping both $\Sigma_1$ and $\Sigma_2$, and a column $y_i(z)$ of the corresponding fundamental solution $Y_{\Sigma_{12}}(z)$. This defines a vector-valued solution which decays exponentially in one of the Stokes sectors $\Sigma_i$ but grows exponentially in the other, which by the uniqueness of the subdominant solution immediately implies (i).

Assume now that we are in the situation of (ii). By the first part the number of Stokes sectors must be even. Consider three Stokes sectors $\Sigma_1,\Sigma_2,\Sigma_3$ in clockwise order. As in the last paragraph, consider an open sector $\Sigma_{12}$ intersecting both $\Sigma_1$ and $\Sigma_2$, and similarly a sector $\Sigma_{23}$ intersecting both $\Sigma_2$ and $\Sigma_3$. We can further assume that $\Sigma_{12}$ intersects $\Sigma_{23}$, and thus write
\[Y_{\Sigma_{23}}(z)=Y_{\Sigma_{12}}(z)\cdot C, \qquad z\in \Sigma_{12}\cap \Sigma_{23}\]
where $C$ is a constant matrix.
 Each of the columns of $Y_{\Sigma_{12}}(z)$ has exponential decay in one of the sectors $\Sigma_1,\Sigma_2$, and hence is a constant multiple of  either $v_1(z)$ or $v_2(z)$. The same argument applies to $Y_{\Sigma_{23}}(z)$, and we conclude that $C$ is diagonal. But since $Q(z)$ is also diagonal, the asymptotic condition \eqref{fl2} 
 implies that $C=1$. 

Continuing around the circle in this way we see that all the distinguished solutions $Y_\Sigma(z)$ for non-Stokes sectors $\Sigma$ agree under analytic continuation and hence define a single-valued solution $Y(z)$ on the punctured disc $\Delta^*$ whose columns are given by constant multiples of $v_1(z)$ and $v_2(z)$. The condition \eqref{fl2} then implies that the matrix-valued function  $Y(z)\cdot \exp(-Q(z))$ extends holomorphically over the origin and returns the identity matrix at $z=0$. Taking the columns of this matrix gives the result.
\end{proof}

\section{The monodromy framed local system}
\label{locale}
In this section we define the monodromy framed local system of  a signed meromorphic projective structure without apparent singularities, and prove the crucial result that it is non-degenerate. We do this by combining  the local analysis of the last section with global properties  of the vector bundle $E$ underlying the corresponding meromorphic oper.

\subsection{Construction of the framed local system}

Consider a meromorphic projective structure $\cP$ on a Riemann surface $S$ having at least one pole. In Section \ref{la} we explained how associate to the pair $(S,\cP)$ a marked bordered surface $(\bS,\bM)$. When $\cP$ is equipped with a signing, and has no apparent singularities, we now explain how to construct a natural framed local system on $(\bS,\bM)$, which we call the \emph{monodromy framed local system} of the pair $(S,\cP)$.

 Let $S^*\subset S$ be the complement of the poles of $\cP$. Thus $\cP$ restricts to give a holomorphic projective structure on $S^*$. As explained in Section \ref{sta}, there is an associated monodromy $G$ local system $\cG$ on the surface $S^*$, whose space of sections over a contractible open subset $U\subset S^*$ is the set $\cG(U)$ of projective charts $\phi\colon U\to \bP^1$ in the projective structure  $\cP$. 
This local system $\cG$ can also be identified with the projective frame bundle of the associated oper $(E,L,\nabla)$. It follows that the $\bP^1$-bundle
\[\cL(U)=\cG(U)\times_G \bP^1,\]
is the projective bundle $\bP(E)$ with its induced connection.

As explained above, if we choose a base projective structure on $S$, and a local chart $z\colon U\to \bC$, the projective charts of $\cP$ defined on $U$ are given by ratios of linearly independent solutions to the vector-valued  differential equation
\begin{equation}
\label{abov}\frac{d}{dz}\vect{s(z)}{t(z)}+\mat{0}{1}{\varphi(z)}{0} \vect{s(z)}{t(z)}=0.\end{equation}
The sections of $\cL$ are then given by one-dimensional spaces of solutions to this equation.

The interior of the surface $\bS^*=\bS\setminus \bP$ can be identified with the complement  $S^*\subset S$ of the poles of $\cP$  considered above. Thus we obtain a $G$ local system $\cG$ on $\bS^*$. The argument above shows that flat sections of the associated $\bP^1$ bundle \eqref{l} can be identified with  nonzero solutions to the equation \eqref{abov} up to scale.

Suppose now that $\cP$ is equipped with a signing and has no apparent singularities. Then we have a choice of sign of the exponent $r(p)$ at each regular pole $p$. This is equivalent, by Proposition~\ref{above}, to a choice of eigenvalue of the monodromy around~$p$. Since this monodromy is not the identity, the associated eigenspace is one-dimensional and defines a section of $\cL$ in a neighborhood of~$p$. On the other hand, Theorem \ref{d} shows that near an irregular pole of $\cP$ there is a unique up-to-scale  solution of \eqref{abov} in each Stokes sector, which gives rise to a flat section of the bundle $\cL$ near the corresponding marked point on $\partial \bS$. Putting everything together gives the required monodromy framed local system on $(\bS,\bM)$.

\subsection{Global description}
\label{ev}

  We now want to prove that the monodromy framed local system constructed in the last section is always non-degenerate. Suppose given a meromorphic projective  structure $\cP$ with poles $\{p_1,\cdots, p_d\}\subset S$ all of which have even order. Let $m_i=2q_i$ be the order of the pole of $\cP$ at the point $p_i$. There are divisors on~$S$ defined as 
\[P=\sum_i p_i, \qquad Q=\sum_i q_i p_i, \qquad D= \sum_i m_i p_i.\]
Let us write $\cP=\cP_0+ \phi$, with $\cP_0$ a holomorphic projective structure, and
\[\phi\in H^0(S,\omega_S^{\tensor 2}(D)),\]
a quadratic differential with polar divisor $D$. The holomorphic projective structure $\cP_0$ has an associated oper $(E,L,\nabla_0)$ well-defined up to tensoring with line-bundles with connection $(M,\nabla_M)$. Here $E$ is a rank 2 vector bundle, and $L$ is a line sub-bundle fitting into a short exact sequence
\[0\lra L\lra E\stackrel{q}{\lra} L\tensor \omega_S^*\lra 0,\]
Since $E$ has a flat connection it has degree 0. But $\omega_S$ has degree $2g-2$ so it follows that $L$ must have degree $g-1$.

 As in Section \ref{meroper} we then introduce a meromorphic connection $\nabla=\nabla_0+\alpha(\phi)$ which in a local basis of flat sections of $\nabla_0$  is given by
\[\nabla+\alpha(\phi)=d+\mat{0}{1}{\varphi(z)}{0} dz.\]
Note that this coincides with \eqref{bl}. We must now understand the meromorphic gauge change leading to \eqref{q} in global terms.
We consider  the usual twisted line bundle \[ L(-Q)=L\tensor_{\O_S} \O_S(-Q),\]
which we view  as a sub-bundle $L(-Q)\subset L$.  Define a subsheaf 
\[F=\big\{s\in E: q(s)\in L(-Q)\tensor \omega_S^*\subset L\tensor \omega_S^*\big\}\subset E.\]
We then obtain a diagram in which the vertical arrows are inclusions
\begin{equation}
\begin{aligned}
\label{exact}\xymatrix{ 
0\ar[r] &L\ar[d]^{\id}\ar[r] &F\ar[d]^{\mu}
 \ar[r] &L(-Q)\otimes\omega_S^*\ar[r]\ar[d]^{\nu}&0\\
0\ar[r] &L \ar[r]  &E \ar[r] &L\otimes\omega_S^*\ar[r]&0.
}
\end{aligned}
\end{equation}

We are interested in meromorphic bundles on $S$ with poles along the divisor $P$ (see for example \cite[Chapter 0, Sections 8 and 14]{Sabbah}), so we introduce the sheaf of rings $\O_S(*P)$ whose local sections are meromorphic functions on $S$ with poles at the points $p_i$. 
The map $\nu$ obviously becomes an isomorphism after base-change along the canonical inclusion $\O_S\to \O_S(*P)$, so the diagram \eqref{exact} shows that the same is true for the map $\mu$. In other words, the holomorphic bundles $F$ and $E$ are different lattices in the same meromorphic bundle.

We claim that the induced meromorphic connection on the holomorphic bundle $F$ has the local description \eqref{q}. Indeed, we considered before, in the proof of Lemma \ref{plob}, a local basis $(u(z),v(z))$ for the bundle $E$, with $v(z)$ spanning the sub-bundle $L$. With respect to this basis the connection $\nabla$ took the form \eqref{opereq}.  The gauge change leading to \eqref{q} corresponds to replacing this basis with $( z^{q} \cdot u(z),v(z))$ which is precisely a basis for $F$.

 \subsection{Non-degeneracy}
 
 The aim of this section is to show that the monodromy of a signed meromorphic projective structure without apparent singularities is a non-degenerate framed local system.

\begin{thm}
\label{nondeg}
Let $(\bS,\bM)$ be a marked bordered surface, and if $g(\bS)=0$ assume that $|\bM|\geq 3$.  Then the monodromy framed local system of a signed meromorphic projective structure without apparent singularities is non-degenerate.
\end{thm}

\begin{proof}
Consider a signed meromorphic projective structure $\cP$ without apparent singularities.   Let $(\cG,\ell)$ be the associated framed local system. We must eliminate the three types of degenerate framed local systems appearing in Definition \ref{rn}. Note that the first one, condition (D1), cannot hold by  the opposedness property Proposition \ref{e}(i).  Condition (D3) is only possible when $\bS$ is a closed surface, so that all singularities of the projective structure are regular. But since we have assumed that $\cP$ has no apparent singularities, condition (D3) then implies that the monodromy representation is reducible and  all puncture monodromies are parabolic. This  contradicts  a result of Faltings \cite[Theorem 7]{Falt}.
Thus it remains to show that condition (D2) does not hold.

Let us first suppose for a contradiction that  we can find two distinct globally defined flat sections  $\ell_1,\ell_2$  of the $\bP^1$ bundle $\cL$ over the punctured surface $\bS^*$ which restrict to give all framings. In particular, by the opposedness property Proposition \ref{e}(i), this implies that all poles have even order. Consider the corresponding meromorphic connection $(F,\nabla)$ of Section \ref{ev}. The sections $\ell_i$ define line subbundles $L_i\subset F$ in the complement of the poles of $\nabla$. We claim that these subbundles extend over the poles and give a global splitting $F=L_1\oplus L_2$.

Let $p\in S$ be a pole of the projective structure. If $p$ is an irregular pole the claim follows immediately from Theorem \ref{e}(ii).
When $p$ is regular, the fact that  both sections $\ell_i$ are well-defined on the surface $\bS^*$ and hence preserved by the monodromy around $p$ shows that this monodromy is diagonalisable. Our assumption that all such monodromies are non-trivial therefore implies that the eigenvalues of the residue matrix do not differ by an integer, so  that the connection is non-resonant, and via a holomorphic gauge transformation we can put it in the form
\[\nabla=d+\frac{A \ dz}{z},\]
with $A$ a constant diagonal matrix. The claim is then clear, since in this gauge the monodromy invariant sections are the standard basis vectors of $\bC^2$.

 The assumption that $\bS$ is not a sphere with fewer than three marked points implies that $\omega_S(Q)$ has positive degree. It follows that in the top line of \eqref{exact} we have
\[\deg  L > \deg L(-Q)\otimes\omega_S^*=\deg L -\deg \omega_S(Q).\] 
This implies that the bundle $F$ is unstable, and that $0\subset L\subset F$ is a Harder-Narasimhan filtration. For basic properties of Harder-Narasimhan filtrations see  for example \cite[Section 5]{LeP}.
 We cannot  have $\deg L_1=\deg L_2$, since this would imply that $F$ were semistable, giving a contradiction. Thus without loss of generality we can assume $\deg L_1>\deg L_2$. But then $0\subset L_1\subset F$ is another Harder-Narasimhan filtration for $F$. By uniqueness of such filtrations it follows that $L_1=L$. This implies that, at least away from the poles,  the sub-bundle $L\subset E$ is invariant under $\nabla$, and  this  contradicts  the oper condition of Remark \ref{snobl}(i).

Suppose  for a contradiction that the monodromy framed local system satisfies condition (D2). As in Remark \ref{remas} (iv) we obtain a permutation representation
\begin{equation}
\label{lablab}\sigma\colon \pi_1(\bS)\to \{\pm 1\}.\end{equation}
If this representation is trivial, the monodromy framed local system satisfies the stronger version of condition (D2) we considered above: there exist two distinct globally defined flat sections  $\ell_1,\ell_2$  of the bundle $\cL$ over $\bS^*$ which restrict to give all framings. Thus we can assume that  $\sigma$ defines an unramified double cover \begin{equation}
\label{done}\pi\colon \tilde{\bS}\to \bS,\end{equation} and setting $\tilde{\bM}=\pi^{-1}(\bM)$ form a marked bordered surface $(\tilde{\bS},\tilde{\bM})$. Contracting the boundary components of $\tilde{\bS}$ gives rise to a double cover  of the Riemann surface $\varpi\colon \tilde{S}\to S$. Note that the  set of branch points of this cover is contained in the set of irregular poles of $\cP$, since the two sections $\ell_1,\ell_2$ are well-defined in a neighbourhood of each puncture. 

The meromorphic projective structure $\cP$ induces a projective structure $\tilde{\cP}=\varpi^*(\cP)$ on the surface $\tilde{S}$, and since $\varpi$ is not branched at the regular singularities, the signing of $\cP$ immediately induces a signing of $\tilde{\cP}$. One can then check that the monodromy framed local system of the signed projective structure $\tilde{\cP}$  is  the pullback of the monodromy framed local system of the signed projective structure $\cP$ along the double cover \eqref{done}. But this pulled-back framed local system satisfies the stronger version of (D2), since by definition of the cover \eqref{done} the flat sections $\ell_1,\ell_2$ are globally-defined on~$\tilde{\bS}$. Applying the argument given above to the signed meromorphic projective structure $\tilde{\cP}$   therefore gives a contradiction.
\end{proof}

\begin{remark}
One could avoid the appeal to Faltings' result if one could prove that the subset of rigidified framed local systems satisfying condition (D3) of Definition \ref{rn} lies in the closure of the subset satisfying condition (D2). Indeed, since the monodromy map of $F$ is holomorphic (as we will soon show), the complement of its image is necessarily closed.
\end{remark}

\subsection{Special cases}
\label{remark}

It is instructive to examine the case of marked bordered surfaces $(\bS,\bM)$ with $g(\bS)=0$ and $|\bM|<3$ which are excluded by  
the statement of Theorem \ref{one}.  In each case, using the standard projective structure on $\bP^1$, we can identify the space of projective structures with the corresponding space of meromorphic quadratic differentials.

There are  two cases when $g(\bS)=0$ and $|\bM|=1$: the once-punctured sphere and the disc with one marked point on the boundary. In these cases it is easy to see that both spaces $\cP(\bS,\bM)$ and $\cX^*(\bS,\bM)$ are empty. The corresponding meromorphic quadratic differentials  would have a single pole of order $m=2$ or 3, but no such differentials exist, since they would define sections of the line bundle
\[\omega_{\bP^1}^{\tensor 2}(m)=\O_{\bP^1}(-4+m),\]
which has negative degree.
On the other hand, framed local systems in these two cases always satisfy condition (D2) of Definition \ref{rn}, since the surface $\bS^*$ is contractible, and there is only one marking $\ell=\ell(p)$.

There are four cases when $g(\bS)=0$ and $|\bM|=2$, two of which are genuinely degenerate and are discussed in Examples \ref{remarks} below. The two remaining cases are the once-punctured disc with one marked point on the boundary and the annulus with one marked point on each boundary component.  The statement of Theorem \ref{nondeg} in fact holds without modification in these cases, and the proof requires only a small additional argument as we now explain.

Since these two surfaces are not closed we just have to check that condition (D2) cannot hold. The stronger  version of this condition   that  we considered in the proof of Theorem \ref{nondeg} cannot hold because the surfaces $(\bS,\bM)$ have boundary components with an odd number of marked points. Thus the monodromy representation \eqref{lablab} must be non-trivial. In the case of the once-punctured disc with one marked point on the boundary this already gives a contradiction, since $\bS$ is contractible. In the other case  the covering surface is an annulus with with two marked points on each boundary component, so the argument  of Theorem \ref{nondeg} applies.

\begin{examples}
\label{remarks}
There are two degenerate cases which need to be excluded from the statement of Theorem \ref{one}. The basic problem with these examples is that the stacks $\cX(\bS,\bM)$ and $\cP(\bS,\bM)$ have non-trivial generic automorphism group. 

\begin{itemize}
\item[(a)]
Suppose that $(\bS,\bM)$ is the twice-punctured sphere. The corresponding space of marked projective structures corresponds to the quadratic differentials $\phi(z)=a dz^{\tensor 2}/z^2$ with $a\in \bC^*$, which have two  poles of order two. Each point has a  $\bC^*$ automorphism group given by rescaling the co-ordinate $z$.  
A framed local system on $(\bS,\bM)$ in this case consists  of a single element of $G$ equipped with a choice of  two (not necessarily distinct) invariant lines. Each of these framed local systems has non-trivial automorphism  group, and all of them are degenerate.
\smallskip

\item[(b)]
Suppose that $(\bS,\bM)$ is the unpunctured disc with two marked points on the boundary.  The corresponding space of meromorphic quadratic differentials consists of the single quadratic differential  $\phi(z)=dz^{\tensor 2}$, which has a single pole of order four at infinity. The resulting marked projective structure has  automorphism group $\bC$ given by translating the co-ordinate~$z$.  A framed local system on $(\bS,\bM)$ in this case consists of a pair of unordered points of $\bP^1$ considered up to the action of $G$. There are two isomorphism classes:  both are degenerate and have non-trivial automorphism groups.
\end{itemize}
\end{examples}

\section{Families of projective structures}

In this section we study families of holomorphic and meromorphic projective structures. This  will be necessary to define complex manifolds parameterising projective structures in the next section.

 \subsection{Families of holomorphic projective structures}
 
 Recall that  a \emph{family of  Riemann surfaces} is defined to be a  holomorphic map of complex manifolds $\pi\colon X\to B$ which is everywhere submersive, and whose fibres $X(b)=\pi^{-1}(b)$ have dimension one.
 A \emph{relative projective atlas} for such a family consists of an open cover $X=\bigcup_{\alpha\in A} U_\alpha$ and a  collection of holomorphic maps $f_\alpha\colon U_\alpha \to \bP^1$, such that if for any $b\in B$ we define \[U_\alpha(b)=U_\alpha\cap X(b), \quad f_\alpha(b)=f_\alpha|_{U_\alpha(b)}\colon U_\alpha(b)\to f_\alpha(U_\alpha(b)),\]  then the data $(U_\alpha(b),f_\alpha(b))$ defines a  projective atlas for the Riemann surface $X(b)$.

Clearly, a relative projective atlas  on a family of Riemann surfaces $\pi \colon X\to B$ induces projective structure $\cP(b)$ on each of the Riemann surfaces $X(b)$.
We define a
 \emph{family of projective structures} on  $\pi\colon X\to B$ to be a collection of projective structures $\cP(b)$ on the fibres $X(b)$ which is induced by a relative projective atlas in this way.

We now consider a relative version of Proposition \ref{fun}. We assume that the map $\pi$ is proper, so that the fibres $X(b)$ are compact.  There is then a vector bundle \[q \colon \cQ(X/B)\to B\] whose fibre over a point $b\in B$ is the vector space of global quadratic differentials \[\cQ(X/B)_b=H^0\big(X(b),\omega_{X(b)}^{\tensor 2}\big).\]

Let  $\cP(X/B)$ denote the the set of pairs $(b,\cP(b))$ consisting of a point $b\in B$ and a projective structure $\cP(b)$  on the corresponding surface $X(b)$. There is an obvious forgetful map \[p\colon \cP(X/B)\to B,\]
and to give a set-theoretic section of this map is the same thing as to give a collection of projective structures $\cP(b)$ on the fibres $X(b)$.

\begin{prop}[\cite{Hub}]\label{projst}Let $\pi\colon X\to B$ be a proper family of Riemann surfaces. Then there is a unique complex manifold structure on the set  $\cP(X/B)$  with the following properties:\begin{itemize}
\item[(i)] The map $p\colon \cP(X/B)\to B$ is an affine bundle for the vector bundle $q \colon \cQ(X/B)\to B$.
\smallskip

\item[(ii)] Holomorphic sections of the affine bundle $p$  correspond precisely to families of projective structures on $\pi\colon X\to B$.
\end{itemize}
\end{prop}

It is possible to show that the complex manifold $\cP(X/B)$ defined by Proposition \ref{projst} has 
a certain universal property and can therefore be viewed as a fine moduli space, but we shall not need this point of view here; details can be found in \cite{Hub}.

\subsection{Relative developing maps}

Suppose $S$ is a Riemann surface equipped with a projective structure $\cP$. If $S$ is simply-connected one can easily show \cite[Lemma 1]{Hub} that there is a global projective chart, or in other words, a holomorphic map
\[f\colon S\to \bP^1\]
which lies in the projective atlas $\cP$. Such a global chart is the same thing as a developing map in this case and is unique up to composition with an element of $G=\Aut(\bP^1)$.
If a discrete group $\Gamma$ acts on $S$ preserving the projective structure, it follows that there is a group homomorphism
\begin{equation}
\label{mondrom}\rho\colon \Gamma\to G, \qquad f(\gamma\cdot x)=\rho(\gamma)\cdot f(x).\end{equation}

Given a projective structure on an arbitrary Riemann surface $S$, we can pull back to get a projective structure on the universal covering surface, invariant under the action of $\Gamma=\pi_1(S)$. A global chart on the cover is nothing but a developing map, as defined before, and the group homomorphism \eqref{mondrom} is the associated monodromy representation
\[\rho\colon \pi_1(S)\to G.\]

Now let $\pi\colon X\to B$ be a family of Riemann surfaces, and assume that the base $B$ is contractible and Stein. If
 $X$ is simply-connected, and $\cP(b)$ is a family of projective structures on the fibres $X(b)$, Hubbard showed \cite[Section 6]{Hub} that there exists  a global relative  projective chart, in other words, a holomorphic map
\[f\colon X\to \bP^1\]
such that for each $b\in B$ the restriction $f(b)\colon X(b)\to \bP^1$ is a chart in the projective structure $\cP(b)$. The map $f$ is called a \emph{relative developing map} and is unique up to composition with a map $B\to G=\Aut(\bP^1)$. Of course each map $f(b)$ is a developing map for the projective structure $\cP(b)$.

Suppose again that a group $\Gamma$ acts on $X$ preserving the map $\pi\colon X\to B$ and the projective structures  $\cP(b)$ on the fibres $X(b)$. Then, having chosen a relative developing map $f$, we obtain a holomorphic map
\[\rho\colon B\times \Gamma \to G, \quad f(\gamma\cdot x)=\rho(\pi(x),\gamma)\cdot f(x).\]
The restrictions $\rho(b,-)\colon \Gamma\to G$ are the representations considered before.

Given an arbitrary family of Riemann surfaces with a  family of projective structures,  we can first restrict to a small open ball in $B$, and then pull back to the universal cover of $X$, before applying the results above. In this way we see that the monodromy representations vary holomorphically with $b\in B$.

\subsection{Families of meromorphic projective structures}
\label{fam}

 Suppose again that $\pi\colon X\to B$ is a proper family of  Riemann surfaces. Let us now suppose that we have chosen $d\geq 1$ disjoint holomorphic sections \[p_i\colon B\to X, \qquad \pi\circ p_i=\id_B,\qquad i=1,\cdots,d,\] so that each Riemann surface $X(b)$ has $d$ ordered marked points $p_i(b)$ on it.
Let us also choose $d$ positive integers $m_i>0$ and introduce  the effective divisor
\[D=\sum_i m_i  D_i, \qquad D_i=p_i(B)\subset X.\]
 It restricts to give a divisor
 $D(b)=\sum_i m_i p_i(b)$ on each surface $X(b)$.
 There is a vector bundle \begin{equation}
\label{disp}q\colon \cQ(X/B;D)\to B\end{equation} whose fibre over a point $b\in B$ is the vector space  \[\cQ(X/B;D)_b=H^0\big(X(b),\omega_{X(b)}^{\tensor 2}(D(b))\big)\] consisting of meromorphic quadratic differentials on $X(b)$ having poles of degree $\leq m_i$ at the points $p_i(b)$ and no other poles. We can form an affine bundle \[\cP(X/B;D)=\cP(X/B)\oplus_{\cQ(X/B)} \cQ(X/B;D)\]
 for this  vector bundle  by performing the construction of Lemma \ref{triv} fibrewise. The following result is then immediate.

\begin{lemma}
The points of $\cP(X/B;D)$ are in bijection with the set of pairs $(b,\cP(b))$ where $b\in B$ and $\cP(b)$ is a meromorphic projective structure on the fibre $X(b)$ having poles of degree $\leq m_i$ at the points $p_i(b)$ and no others.
\end{lemma}

We define a  \emph{family of meromorphic projective structures} on  $\pi\colon X\to B$ relative to the divisor $D$ to be a holomorphic section of the affine bundle
\[p\colon \cP(X/B;D)\to B.\]
When $D=0$ this reduces to the previous notion of a family of holomorphic projective structures by Proposition \ref{projst} (ii).

\subsection{Local description}
 We shall need the following local description of  a family of meromorphic projective structures. Take notation as above, and let $\Delta\subset \bC$ denote the unit disc.

\begin{prop}
\label{wall}
Suppose given a family of meromorphic projective structures $\cP(b)$ on the family $\pi\colon X\to B$ relative to the divisor $D$. Given a choice of one of the divisors $D_i$, there exists an open subset $U\subset X$ and an isomorphism $h$ fitting into the diagram 
\[\xymatrix{ U \ar[rr]^{h}_{\isom} \ar[dr]_{\pi|_U}&& V\times\Delta \ar[dl]^{p_1} \\&V}\]
where $V=\pi(U)$, 
such that $h^{-1}(V\times \{0\})=D_i\cap \pi^{-1}(V)$. Moreover, there exists a holomorphic function $\varphi(b,z)$ on $V\times \Delta$ such that the family of meromorphic projective structures $h(\cP(b))$ induced on $\Delta$ is given by  ratios of solutions to the equation
\[y''(z)+z^{-m_i}\cdot  \varphi(b,z) y(z)=0.\]
\end{prop}

\begin{proof}
Replacing $B$ by an open ball we can take a family of holomorphic projective structures on $\pi\colon X\to B$.  Replace $X$ by a contractible open neighbourhood of the divisor $D_i$. We  then take a relative developing map $f\colon X\to \bP^1$, which we can assume satisfies $f(D_i)=\{0\}$ and $\Delta\subset f(X)$. Set $U=f^{-1}(\Delta)$ and define $h(x)=(\pi(x),f(x))$. Then $h$ is easily seen to have non-degenerate derivative, and so shrinking $B$ and $U$ again if necessary we can ensure that $h$ is  an isomorphism.

Pushing the family of meromorphic projective structures across this isomorphism gives a family of meromorphic projective structures on $\Delta$ relative to the divisor $m_i\cdot [0]$. Subtracting the trivial family of holomorphic projective structures gives a family of meromorphic quadratic differentials on $\Delta$ which can be written in the form $z^{-m_i} \cdot \varphi(b,z) dz^{\tensor 2}$ with $\varphi(b,z)$ holomorphic,  and so the result follows.
\end{proof}

We will also need a family version of the statement that meromorphic projective structures induce holomorphic projective structures on the complement of their poles.

\begin{lemma}
A  family of meromorphic projective structures on the family $\pi\colon X\to B$ relative to the divisor $D$ induces a family of holomorphic projective structures on the family $\pi\colon X\setminus D \to B$.
\end{lemma}

\begin{proof} Fibrewise this is obvious: a projective structure on  the fibre $X(b)$ having poles only at the points    $p_i(b)$ induces a holomorphic projective structure on the complement $X(b)\setminus D(b)$. One must show the existence of relative atlases. This follows exactly as in the proof of \cite[Proposition~1]{Hub}.
\end{proof}

\section{Moduli spaces of projective structures}
\label{msps}

In this section we introduce complex manifolds $\Proj(\bS,\bM)$ parameterising marked meromorphic projective structures. We also introduce the variants $\Proj^\circ(\bS,\bM)$ and $\Proj^*(\bS,\bM)$ mentioned in the introduction.
We then prove that the generalized monodromy map
\[F\colon \Proj^*(\bS,\bM)\to \cX^*(\bS,\bM)\]
 can be locally lifted to give a holomorphic map into the complex variety $X^*(\bS,\bM)$. 
The proof of Theorem \ref{one} will be completed in Section \ref{final} when we use Fock-Goncharov co-ordinates to show that the quotient $\cX^*(\bS,\bM)$ is a complex manifold.

\subsection{Moduli of marked projective structures}

Recall from Section \ref{mbs} the basic definitions concerning marked bordered surfaces. Let $(\bS,\bM)$ be a fixed marked bordered surface determined by a genus $g$ and a non-empty collection of non-negative integers $\{k_1,\cdots, k_d\}$. For each $i=1,\cdots,d$ we set $m_i=k_i+2$. Recall from Section \ref{la} that these integers give the pole orders of a meromorphic  projective structure with the associated marked bordered surface $(\bS,\bM)$.

Consider a meromorphic projective structure $\cP$ on a compact Riemann surface $S$.  By a marking of the pair $(S,\cP)$ by the surface $(\bS,\bM)$ we mean an isotopy class of isomorphisms between  $(\bS,\bM)$ and the marked bordered surface associated to the pair $(S,\cP)$.
 Clearly, the set of markings of a given pair $(S,\cP)$ by the surface $(\bS,\bM)$ is either empty or is a torsor for the mapping class group $\MCG(\bS,\bM)$.

\begin{lemma}
\label{loc}
Take notation as in Section \ref{fam}. Then there is a dense open subset of  the manifold $\Proj(X/B;D)$ parameterising projective structures with poles of orders exactly  $\{m_1,\cdots, m_d\}$. There is also a principal $\MCG(\bS,\bM)$ bundle over this open subset  whose fibre at a point $(b,\cP(b))$ is the set of markings of the pair $(X(b),\cP(b))$.
\end{lemma}

\begin{proof}
After choosing a family of holomorphic projective structures on $\pi\colon X\to B$, the dense open subset corresponds to the subset of $\cQ(X/B;D)$ consisting of quadratic differentials whose zeroes are disjoint from the divisors $D_i$. This is easily seen to be open and dense.
The result then follows immediately once one knows that the asymptotic horizontal directions of these quadratic differentials vary continuously. But these directions are determined by the leading coefficient, which certainly varies continuously.
\end{proof}

Let us consider triples $(S,\cP,\theta)$, where $S$ is a Riemann surface equipped with a meromorphic projective structure $\cP$, and  $\theta$ is a marking of the pair $(S,\cP)$ by the surface $(\bS,\bM)$. Two such triples $(S_i,\cP_i,\theta_i)$ are considered to be equivalent if there is a biholomorphism  $f\colon S_1\to S_2$ between the underlying Riemann surfaces, which preserves the projective structures $\cP_i$ and commutes with the markings $\theta_i$ in the obvious way.

\begin{prop}
Assume that if $g(\bS)=0$ then $|\bM|\geq 3$. Then the set $\cP(\bS,\bM)$ of equivalence classes of marked projective structures  on $(\bS,\bM)$ has the natural structure of a complex manifold of dimension
\[n = 6g-6+\sum_i (k_i+3).\]
\end{prop}

\begin{proof}
The definition of a marked bordered surface ensures that $d\geq 1$.  Let us also assume for now that if $g=0$ then $d\geq 3$. Let $B=\cM(g,d)$ denote the moduli stack of compact Riemann surfaces of genus $g$ with $d$ ordered marked points. Under our assumptions it is a smooth Deligne-Mumford stack, of dimension $3g-3+d$, and can thus be viewed as a complex orbifold. Readers squeamish about stacks can replace it by its universal cover, which under our assumptions is a complex manifold, the Teichm{\"u}ller space $\cT(g,d)$ of marked, pointed Riemann surfaces.

There is a universal curve $\pi\colon X\to B$ with $d$ disjoint sections $p_1,\cdots, p_d$. We can apply the construction of Section \ref{fam} to obtain a space $\cP(X/B;D)$ whose points are in bijection with the set of isomorphism classes of pairs $(S,\cP)$ consisting of a compact Riemann surface $S$ of genus $g$ together with a meromorphic projective structure $\cP$ having $d$ poles of orders $m_i\leq k_i+2$. Riemann-Roch tells us that the vector bundle \eqref{disp} has rank
\[3g-3 + \sum_i (k_i+2),\]
so that the space $\cP(X/B;D)$ has the claimed dimension. 

Using Lemma \ref{loc} we can then pass to a covering space $\cP(\bS,\bM)$ of an open subset whose points consist of such pairs together with a marking. If one is dealing with Teichm{\"u}ller space $\cT(g,d)$ one should insist that the marking of $(S,\cP)$ by $(\bS,\bM)$ is compatible in the obvious way with the marking of the underlying curve $S$ by the closed surface obtained by blowing-down all boundary components of $\bS$. In any case, the resulting space is indeed a manifold rather than an orbifold, because the same is true of $\cT(g,d)$: although the surface $S$ may have automorphisms, none of these fix the marking of $S$, let alone the projective structure $\cP$ and its associated marking.

The cases when $g=0$ and $d\leq 2$ can be treated in the same way if one is happy to consider Artin stacks, but they can also be understood directly. The relevant projective structures are obtained from the standard one on $\bP^1$ by adding a meromorphic differential of the form $p(z) dz^{\tensor 2}$ with $p(z)$ a Laurent polynomial. The assumption $|\bM|\geq 3$ implies that $p(z)$ is not constant, so the only possible automorphisms are rescalings of $z$ by roots of unity, and these are easily seen to act non-trivially on the markings, which consist of a labelling of the asymptotic directions at infinity. Thus although the spaces $\cM(g,d)$ and $\cP(X/B;D)$ are  smooth Artin stacks,  the cover  $\cP(\bS,\bM)$ is again a manifold. 
\end{proof}

We will often loosely refer to the points of the space $\cP(\bS,\bM)$ as parameterising marked projective structures on the marked bordered surface $(\bS,\bM)$.

\subsection{Monodromy and signed projective structures}

Let us fix a marked bordered surface $(\bS,\bM)$ and choose a base-point $x\in \bS^*$.  Let us write $B= \cP(\bS,\bM)$.  Taking monodromy of the meromorphic projective structures gives a family of representations
\[\rho_b\colon \pi_1(\bS^*,x)\to G,\]
indexed by the points $b\in B$, and well-defined up to overall conjugation by a holomorphic function $g\colon B\to G$.
For each $\gamma\in \pi_1(\bS^*,x)$, the resulting maps
\[\rho(\gamma)\colon B\to G, \qquad b\mapsto \rho_b(\gamma)\]
are holomorphic. Indeed, this is a local question, so we can replace $B$ by a small open ball, and then the existence of relative developing maps discussed above gives the result.

\begin{lemma}
There is an open subset 
\[\cP^\circ(\bS,\bM)\subset \cP(\bS,\bM)\]
consisting of marked projective structures without apparent singularities on $(\bS,\bM)$.
\end{lemma}

\begin{proof}
This is immediate from the above discussion. Indeed,  in terms of the  elements $\delta_p\in \pi_1(\bS^*,x)$ defined in Section \ref{framedlocalsystems}, the subset $\cP^\circ(\bS,\bM)$  is defined by the condition that $\rho(\delta_p)\neq 1\in G$ for all $p\in \bP$, and this is clearly open.
\end{proof}

Recall the notion of a signed projective structure from Section \ref{sig}.

\begin{prop}
There is a complex manifold $\cP^*(\bS,\bM)$ equipped with a finite map
\begin{equation}
\label{jih}\cP^*(\bS,\bM)\to \cP^\circ(\bS,\bM)\end{equation}
whose points are in bijection with signed marked projective structures without apparent singularities on $(\bS,\bM)$.
\end{prop}

\begin{proof}
Note first that there is a map
\begin{equation*}
a\colon\cP(\bS,\bM)\to \bC^{\bP}\end{equation*}
sending a marked projective structure to the leading coefficient at each of the punctures  $p\in \bP$. 
 This map  is holomorphic by Proposition \ref{wall},  and we can therefore pass to the branched  cover
 \begin{equation}
\label{lala}\Proj^*(\bS,\bM)\to \Proj(\bS,\bM)\end{equation}
 obtained by choosing a sign of the exponent \eqref{resnot}. This cover is smooth providing $a$ is a  submersion, which holds by the proof of  \cite[Lemma 6.1]{BS}.
The cover \eqref{lala} is clearly finite of degree $2^{|\bP|}$, and its points  are in bijection with signed marked projective structures.
\end{proof}

\subsection{Monodromy map is holomorphic}

Fix a marked bordered surface $(\bS,\bM)$, and if $g(\bS)=0$ assume that $|\bM|\geq 3$.
Sending a signed meromorphic projective structure without apparent singularities to its associated framed local system gives a generalised monodromy map
\[F\colon \cP^*(\bS,\bM)\to \cX^*(\bS,\bM).\]
We must show that this map is holomorphic. More precisely, let us take a small open ball $B\subset \cP^*(\bS,\bM)$. Via the map \eqref{jih} the ball $B$ parameterises a family of meromorphic projective structures. Let us choose a relative developing map. We can then lift the map $F$ so as to land in the algebraic variety
\[X(\bS,\bM)=\Big\{\rho\in \Hom_{\Grp}(\pi_1(\bS^*,x),G), \lambda\in \Hom_{\Set}(\bM,\bP^1):  \rho(\delta_p)(\lambda(p))=\lambda(p)\text{ for all }p\in \bP\Big\}.\]
It is the resulting locally-defined map $B\to X(\bS,\bM)$ that we will prove is holomorphic. The proof of Theorem \ref{one} will be completed in Section \ref{final} when we use Fock-Goncharov co-ordinates to show that the quotient $\cX^*(\bS,\bM)$ is a complex manifold.

We have already showed that the elements $\rho_b(\gamma)$ vary holomorphically with $b\in B$. It remains to prove that the elements $\lambda_b(p)\in \bP^1$ also vary holomorphically. There are two cases $p\in \bP$ and $p\in \bM\setminus\bP$ corresponding to regular and irregular singularities.

Considering the regular case first, suppose $p\in \bP$ is a puncture. Then the monodromy $M=\rho_b(\delta_p)\in G$ varies holomorphically with $b\in B$, and we can lift it to an element of $\SL_2(\bC)$. The eigenvalues $\mu_\pm$ of the matrix $M$ also vary holomorphically. It then follows that the corresponding eigenvectors also vary holomorphically, since they are defined by $(M-\mu_\pm)(v_\pm)=0$.

Consider now the irregular case, so that $p\in \bM$ is a marked point on a boundary component of~$\bS$. The corresponding $\lambda(p)$ is defined by a subdominant solution.  Using Proposition \ref{wall} we reduce to the following statement.

\begin{thm} Suppose given a  family of potentials depending on a parameter $t$ of the form $\varphi(z;t)=z^{-k} \cdot h(z,t)$ with $h(z,t)$ holomorphic. Then in the context of Theorem \ref{d} one can find subdominant solutions $y(z;t)$ which vary holomorphically with $t$.
\end{thm}

\begin{proof}
This follows immediately once one knows that the distinguished solutions we used in the proof of Theorem \ref{d} vary holomorphically with parameters. This statement can be found for example in  \cite[Lemma 7]{Boalch}.
\end{proof}

\section{Fock-Goncharov co-ordinates}
\label{final}

Fock and Goncharov defined a system of birational co-ordinate charts on the stack $\cX(\bS,\bM)$ of framed $G$-local systems parameterised by ideal triangulations of the surface $(\bS,\bM)$. We give a mild extension of their results to cover the case of tagged triangulations. We also prove that for any point of the open substack $\cX^*(\bS,\bM)$ of non-degenerate framed local systems there exists some tagged triangulation such that the corresponding co-ordinates are well-defined and nonzero.

\subsection{Ideal triangulations}
\label{tri}

This section contains some standard material on ideal triangulations of marked bordered surfaces. A more careful treatment can be found in \cite{FST}. Throughout, $(\bS,\bM)$ denotes a marked bordered surface. 

An \emph{arc} in $(\bS,\bM)$ is a smooth path $\gamma$ in  $\bS$ connecting  points of $\bM$, whose interior lies in the open subsurface $\bS\setminus (\bM\cup \partial \bS)$, and which has no self-intersections in its interior.  We moreover insist that $\gamma$ not be homotopic, relative to its endpoints, to a single point, or to a path in $\partial \bS$ whose interior contains no points of $\bM$. Two arcs are considered to be equivalent if they are related by a homotopy through such arcs. A path that connects two marked points and lies entirely on the boundary of $\bS$ without passing through a third marked point is called a \emph{boundary segment}.

An \emph{ideal triangulation} of  $(\bS,\bM)$ is defined to be a maximal collection of equivalence classes of  arcs for which it is possible to find representatives whose interiors are pairwise disjoint. We refer to the arcs and boundary segments as the \emph{edges} of the triangulation. An example of an ideal triangulation of a disc with five marked points on its boundary is depicted in Figure \ref{below}. Note that it consists of just two arcs; the boundary segments are not considered to be arcs of the triangulation.

\begin{figure}[ht]
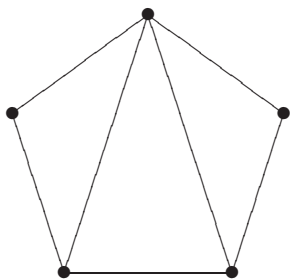

\begin{center}
\[
\xy /l1.5pc/:
{\xypolygon5"A"{~:{(-3,0):}}};
{"A5"\PATH~={**@{-}}'"A2"'"A4"};
"A1"*{\bullet};
"A2"*{\bullet};
"A3"*{\bullet};
"A4"*{\bullet};
"A5"*{\bullet};
\endxy
\]
\end{center}
\caption{A triangulation of a disc with five marked points.}\label{below}
\end{figure}

A \emph{triangle} of an ideal triangulation $T$ is the closure in $\bS$ of a connected component of the complement of all arcs of $T$. Each triangle is topologically a disc, containing either two or three distinct edges of the triangulation. A triangle with just two distinct edges is called a \emph{self-folded triangle} (see Figure~\ref{fig:selffolded}).
 
\begin{figure}[ht]
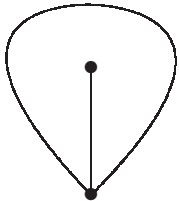
 \begin{center}
\[
\xy /l0.5pc/:
(0,-4)*{}="N"; 
(0,8)*{}="S"; 
(0,0);"S" **\dir{-}; 
(0,0)*{\bullet}; 
"S"*{\bullet}; 
"S";"N" **\crv{(-12,-4) & (0,-4)}; 
"S";"N" **\crv{(12,-4.5) & (0,-4)}; 
\endxy
\]
\caption{A self-folded triangle.\label{fig:selffolded}}
\end{center} \end{figure}
 
The \emph{valency} of a  puncture $p\in \bP$ with respect to a  triangulation $T$ is the number of half-edges of~$T$ that are incident with it; a puncture has valency 1 precisely if it is contained in the interior of a self-folded triangle.

Two ideal triangulations $T_1$ and $T_2$ are  related by a \emph{flip} if  they are distinct and there are edges $e_i\in T_i$  such that $T_1\setminus \{e_1\}=T_2\setminus \{e_2\}$ (see Figure~\ref{flippy}). Note that neither $e_1$ nor $e_2$ is the interior edge of a self-folded triangle. Conversely, if $e$ is not the interior edge of a self-folded triangle, it is contained in exactly two triangles of $T$, and  there is a unique ideal triangulation which is the flip of~$T$ along $e$.

\begin{figure}[ht]
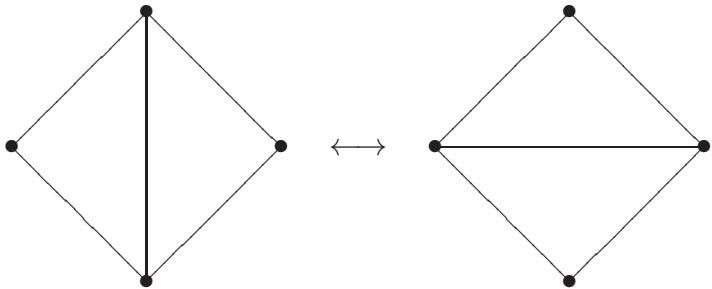

\begin{center}
\[
\xy /l1.5pc/:
{\xypolygon4"A"{~:{(2,2):}}};
{\xypolygon4"B"{~:{(2.5,0):}~>{}}};
{\xypolygon4"C"{~:{(0.8,0.8):}~>{}}};
{"A1"\PATH~={**@{-}}'"A3"};
"A1"*{\bullet};
"A2"*{\bullet};
"A3"*{\bullet};
"A4"*{\bullet};
\endxy
\quad
\longleftrightarrow
\quad
\xy /l1.5pc/:
{\xypolygon4"A"{~:{(2,2):}}};
{\xypolygon4"B"{~:{(2.5,0):}~>{}}};
{\xypolygon4"C"{~:{(0.8,0.8):}~>{}}};
{"A2"\PATH~={**@{-}}'"A4"};
"A1"*{\bullet};
"A2"*{\bullet};
"A3"*{\bullet};
"A4"*{\bullet};
\endxy
\]
\end{center}
\caption{Flip of a triangulation.\label{flippy}}
\end{figure}

Any two ideal triangulations of a surface $(\bS,\bM)$ can be related by a sequence of such flips \cite[Prop. 3.8]{FST}. Thus any ideal triangulation $(\bS,\bM)$ has the same number of arcs, namely the number $n$ of \eqref{n}.  In the cases when $n<0$ the surface has no ideal triangulation. We will write $J=J^T$ for the set of arcs of an ideal triangulation $T$.

There is an $n\times n$ integer matrix encoding the combinatorics of an ideal triangulation~$T$ of a marked bordered surface. For each arc $j$ of the triangulation we denote by $\pi_T(j)$ the arc defined as follows: if $j$ is the interior edge of a self-folded triangle we let $\pi_T(j)$ be the encircling edge, and we define $\pi_T(j)=j$ otherwise. For each non-self-folded triangle $t$ of $T$ we define a number $\varepsilon_{ij}^t$ by the following rules:
\begin{itemize}
\item[(i)] $\varepsilon_{ij}^t=+1$ if $\pi_T(i)$ and $\pi_T(j)$ are sides of $t$ with $\pi_T(j)$ following $\pi_T(i)$ in the counterclockwise order defined by the orientation.
\item[(ii)] $\varepsilon_{ij}^t=-1$ if the same holds with the clockwise order.
\item[(iii)] $\varepsilon_{ij}^t=0$ otherwise.
\end{itemize}
Then the $(i,j)$ element of the \emph{exchange matrix} associated to $T$ is defined as the sum 
\[
\varepsilon_{ij}=\sum_t\varepsilon_{ij}^t
\]
over all non-self-folded triangles of~$T$.

 \subsection{Construction of coordinates}

Let $(\mathbb{S},\mathbb{M})$ be a marked bordered surface. Let us  equip the  surface $\bS\setminus\bM$ with a complete, finite area hyperbolic metric with totally geodesic boundary. Then the universal cover of $\bS\setminus\bM$ can be identified with a subset of the hyperbolic plane~$\mathbb{H}$ with totally geodesic boundary. The deleted marked points on~$\mathbb{S}$ give rise to a set of points on the boundary~$\partial\mathbb{H}$. We call this the \emph{Farey set} and denote it by $\mathcal{F}_\infty(\bS,\bM)$. The action of $\pi_1(\bS^*)$ by deck transformations on the universal cover gives rise to an action of $\pi_1(\bS^*)$ on the Farey set. See \cite[Section 1.3]{IHES}  for more details and a picture. 

It will be useful to reformulate the definition of a framed local system in terms of the Farey set.

\begin{lemma}
\label{lem:equivalentdefinition}
A point of $X(\mathbb{S},\mathbb{M})$ is the same as a pair $(\rho,\psi)$ where $\rho:\pi_1(\bS^*)\rightarrow G$ is a group homomorphism and $\psi:\mathcal{F}_\infty(\mathbb{S},\mathbb{M})\rightarrow\mathbb{P}^1$ is map from the Farey set into~$\mathbb{P}^1$ such that 
\[
\psi(\gamma c)=\rho(\gamma)\psi(c)
\]
for any $\gamma\in\pi_1(\bS^*)$ and $c\in\mathcal{F}_\infty(\mathbb{S},\mathbb{M})$.
\end{lemma}

\begin{proof}
Fix a basepoint $x\in\bS^*$ and for each point $p\in\mathbb{M}$ an arc $\beta_p$ connecting $x$ to~$p$ whose interior lies in $\bS^*$. We have seen that these choices determine a bijection between $X(\mathbb{S},\mathbb{M})$ and the set of pairs $(\rho,\lambda)$ where $\rho:\pi_1(\mathbb{S}^*)\rightarrow G$ is a group homomorphism and $\lambda:\mathbb{M}\rightarrow\mathbb{P}^1$ is a map of sets such that 
\[
\rho(\delta_p)(\lambda(p))=\lambda(p)
\]
for every $p\in\mathbb{P}$ where $\delta_p$ is the homotopy class of loops surrounding~$p$ as before. Let $\tilde{x}$ be a point in the closure of the hyperbolic plane that projects to the basepoint $x$. Each curve $\beta_p$ lifts to a unique curve connecting $\tilde{x}$ to a point $c_p\in\mathcal{F}_\infty(\mathbb{S},\mathbb{M})$, and we set $\psi(c_p)=\lambda(p)$. The points $c_p$ obtained in this way form a complete set of representatives for the orbits of $\mathcal{F}_\infty(\mathbb{S},\mathbb{M})$ under the action of $\pi_1(\bS^*)$. Thus we can extend this construction and define a map $\psi:\mathcal{F}_\infty(\mathbb{S},\mathbb{M})\rightarrow\mathbb{P}^1$ by insisting that $\psi(\gamma c)=\rho(\gamma)\psi(c)$ for all $\gamma\in\pi_1(\bS^*)$ and $c\in\mathcal{F}_\infty(\mathbb{S},\mathbb{M})$.
\end{proof}

Fix an ideal triangulation $T$ of $(\mathbb{S},\mathbb{M})$. We can lift this triangulation to a collection of geodesic arcs in~$\mathbb{H}$ decomposing the universal cover into triangular regions. The endpoints of these arcs in~$\partial\mathbb{H}$ are identified with the points of~$\mathcal{F}_\infty(\mathbb{S},\mathbb{M})$. Thus, if we are given a point of $X(\mathbb{S},\mathbb{M})$, we can use the map $\psi$ to assign points of~$\mathbb{P}^1$ to the endpoints of each geodesic arc.

\begin{definition}
\label{def:fockgoncharovcoordinates}
Let $\mu=(\rho,\psi)$ be a general point of $X(\mathbb{S},\mathbb{M})$. To any arc $j$ of~$T$, we associate a number $X_j\in\mathbb{C}^*$ as follows:
\begin{itemize}
\item[(i)] Suppose $j$ is not the interior edge of a self-folded triangle. Let $\tilde{j}$ be a lift of $j$ to the universal cover. Then there are two triangles of the triangulation that share the side~$\tilde{j}$, and these form an ideal quadrilateral in~$\mathbb{H}$. Let $c_1$, $c_2$, $c_3$, and~$c_4$ be the vertices of this quadrilateral in the counterclockwise order so that the arc $\tilde{j}$ joins the vertices $c_1$ and~$c_3$. For each~$i$, let us define $z_i=\psi(c_i)$. Then we define $X_j$ as the cross ratio 
\[
X_j=\frac{(z_1-z_2)(z_3-z_4)}{(z_2-z_3)(z_1-z_4)}.
\]
Note that there are two ways of ordering the points $c_i$, and they give the same value for the cross ratio.

\item[(ii)] Suppose $j$ is the interior edge of a self-folded triangle. Let $k$ be the loop of this self-folded triangle. Using the same construction described in part~(i), we can associate cross ratios $Y_j$ and $Y_k$ to the arcs $j$ and~$k$, respectively. Then the number $X_j$ is defined by $X_j=Y_jY_k$.
\end{itemize}
\end{definition}

In this way, we associate to a general point $\mu$ a tuple of numbers $X_j\in\mathbb{C}^*$ indexed by arcs $j$ of the ideal triangulation $T$. Taken together, these define a rational map $X_T:X(\mathbb{S},\mathbb{M})\dashrightarrow(\mathbb{C}^*)^n$. This map is invariant under the action of~$G$ on $X(\mathbb{S},\mathbb{M})$ and so we have a rational map 
\[
\mathcal{X}(\mathbb{S},\mathbb{M})\dashrightarrow(\mathbb{C}^*)^n.
\]
We claim that this is a birational equivalence. More precisely, we have the following statement.

\begin{lemma}
For any ideal triangulation $T$ of $(\mathbb{S},\mathbb{M})$, there exists a regular map 
\[
\iota_T:(\mathbb{C}^*)^n\rightarrow X(\mathbb{S},\mathbb{M})
\]
such that $X_T\circ\iota_T=\mathrm{id}$.
\end{lemma}

\begin{proof}
Suppose we are given a collection of nonzero complex numbers $X_j$ indexed by the arcs of~$T$. We replace these by new numbers $Y_j$ as follows:
\begin{itemize}
\item[(i)] Set $Y_j=X_j$ if $j$ is not the interior edge of a self-folded triangle.
\item[(ii)] Set $Y_j=X_j/X_k$ if $j$ is the interior edge of a self-folded triangle and $k$ is the loop of this self-folded triangle.
\end{itemize}
Consider an ideal triangle $\tilde{t}_0$ in~$\mathbb{H}$ which projects to an ideal triangle $t_0$ on the surface~$\mathbb{S}$. We assign three arbitrary distinct points of $\bP^1$ to the vertices of $\tilde{t}_0$. Next consider an ideal triangle $t_1$ which shares an edge $j$ with~$t_0$ in the ideal triangulation. There is an ideal triangle $\tilde{t}_1$ adjacent to $\tilde{t}_0$ in the universal cover which projects to~$t_1$, and we have assigned points of $\bP^1$ to two of its vertices. We can assign a point of~$\bP^1$ to the remaining vertex in such a way that the resulting cross ratio is the number $Y_j$. Continuing in this way, we construct a function $\psi:\mathcal{F}_\infty(\mathbb{S},\mathbb{M})\rightarrow\bP^1$. To construct a monodromy representation, let $\gamma\in\pi_1(\bS^*)$ and let $\tilde{t}$ be an ideal triangle in the universal cover. The element $\gamma$ acts as a deck transformation on the universal cover, taking the ideal triangle $\tilde{t}$ to some other ideal triangle $\tilde{t}'$. We define $\rho(\gamma)$ to be the unique element of $G$ that takes the points of $\bP^1$ assigned to the vertices of~$\tilde{t}$ to the corresponding points assigned to the vertices of $\tilde{t}'$. This defines a homomorphism $\rho:\pi_1(\bS^*)\rightarrow G$, and the pair $(\psi,\rho)$ is a point of $X(\mathbb{S},\mathbb{M})$. It is easy to check that this defines a regular map to $X(\mathbb{S},\mathbb{M})$ and that $(\psi,\rho)$ has the numbers $X_j$ as co-ordinates.
\end{proof}

\subsection{Signed and tagged triangulations}
\label{flops}
 
Let $(\bS,\bM)$ be a marked bordered surface. A \emph{signed triangulation} of $(\bS,\bM)$ is a pair $(T,\epsilon)$ consisting of an ideal triangulation $T$ and a function \[\epsilon\colon \bP\to \{\pm 1\}.\]
Two signed triangulations $(T_i,\epsilon_i)$ are  considered to be equivalent  if the underlying ideal triangulations $T_i$ are the same and the signings $\epsilon_i$  differ only at punctures  $p\in \bP$ of valency one. Equivalence classes of signed triangulations are called \emph{tagged triangulations}.

Suppose $\tau$ is a tagged triangulation which is represented by a signed triangulation $(T,\epsilon)$. By a \emph{tagged arc} of~$\tau$ we mean an arc of the ideal triangulation~$T$. Let $(T,\epsilon')$ be another signed triangulation where $\epsilon'$ differs from $\epsilon$ at a single puncture $p$ where the triangulation $T$ has valency one. Let $j$ be the unique edge of $T$ which is incident to this puncture $p$, and let $k$ be the encircling edge. Then the tagged arc represented by $j$ in $(T,\epsilon)$ is considered to be equivalent to the tagged arc represented by $k$ in the other signed triangulation $(T,\epsilon')$.

If $\Tri(\bS,\bM),  \Tri_\pm(\bS,\bM),  \Tri_{\bowtie}(\bS,\bM)$ denote the sets of ideal, signed and tagged  triangulations of $(\bS,\bM)$ respectively, there is a diagram 
\begin{equation}
\begin{aligned}
\label{triangulationsdiagr}
\xymatrix{ 
&\Tri_{\pm}(\bS,\bM)\ar[d]^{q} \\
\Tri(\bS,\bM) \ar[r]_{i} \ar[ur]^{j}&\Tri_{\bowtie}(\bS,\bM)}
\end{aligned}
\end{equation}
where $q$ is the obvious quotient map and the  arrows $i$ and $j$ are embeddings obtained by  considering an ideal triangulation as a signed, and hence a tagged triangulation,  using the signing $\epsilon\equiv +1$. 
 
The flipping  operation extends to signed triangulations in the obvious way: we flip the underlying triangulation, keeping the signs constant. We say that two tagged triangulations are related by a flip if they can be represented by signed triangulations which differ by a flip.

The sets appearing in the diagram \eqref{triangulationsdiagr} can be considered as graphs, with two (ideal, signed, tagged) triangulations being connected by an edge if they differ by a flip.  The maps in the diagram then become maps of graphs. The important point is that, unlike the graph $\Tri(\bS,\bM)$ of ideal triangulations, the graph $\Tri_{\bowtie}(\bS,\bM)$ of tagged triangulations  is $n$-regular.

It is well-known that any two ideal triangulations of $(\bS,\bM)$ are related by a finite chain of flips; thus the graph $\Tri(\bS,\bM)$ is always connected \cite[Prop. 3.8]{FST}. The graph $\Tri_{\bowtie}(\bS,\bM)$ is also connected, except  when $(\bS,\bM)$ is a closed surface with a single puncture  $p\in \bP$, in which case the graph   $\Tri_{\bowtie}(\bS,\bM)$ has two connected components corresponding to the two possible choices of signs $\epsilon(p)$ \cite[Prop. 7.10]{FST}.

\subsection{Fock-Goncharov co-ordinates for tagged triangulations}

Let $(\bS,\bM)$ be a marked bordered surface and consider a framed local system on $(\bS,\bM)$. Near each puncture $p\in \bP$ there is a chosen flat section $\ell(p)$ of the associated $\bP^1$ bundle $\cL$. If the monodromy of the local system around this puncture has distinct eigenvalues then there are exactly two possible choices for $\ell(p)$, and we can consider the birational morphism which exchanges these two choices.

\begin{lemma}
There is a natural birational action of the group $\{\pm 1\}^{\bP}$ on the stack $\cX(\bS,\bM)$ of framed local systems which fixes the underlying local systems and exchanges the two generically possible choices of framings at the punctures.
\end{lemma}

\begin{proof}
Consider the variety $X(\bS,\bM)$ of rigidified framed local systems. There is an open subset where all puncture monodromies are semi-simple with distinct eigenvalues. On this open subset the operation of exchanging the possible eigenlines at a given puncture gives an algebraic involution.
\end{proof}

The group $\{\pm 1\}^{\bP}$ also acts on the set of signed and tagged triangulations in the obvious way. A signed triangulation $(T,\epsilon)$ can then be viewed as the result of acting on the ideal triangulation $T$ by the signing $\epsilon$ viewed as an element of the group  $\{\pm 1\}^{\bP}$.

\begin{defn}
The Fock-Goncharov co-ordinate of a framed local system $(\cG,\phi)$ on $(\bS,\bM)$ with respect to an arc $j$ of the signed triangulation $(T,\epsilon)$ is defined to be the Fock-Goncharov co-ordinate with respect to $j$ of the framed local system obtained by applying the group element $\epsilon\in\{\pm 1\}^{\bP}$ to the framed local system $(\cG,\phi)$, whenever this quantity is well-defined.
\end{defn}

We thus obtain birational maps 
\begin{equation}
\label{bira}
X_{(T,\epsilon)}\colon \mathcal{X}(\mathbb{S},\mathbb{M})\dashrightarrow(\mathbb{C}^*)^n
\end{equation}
for any signed triangulation.

\begin{lemma}
The Fock-Goncharov co-ordinate of a framed local system with respect to an arc of a signed triangulation depends only on the underlying tagged arc.
\end{lemma}

\begin{proof}
Let $T$ be an ideal triangulation of $(\mathbb{S},\mathbb{M})$, and suppose $p\in\mathbb{P}$ is a puncture which is incident to exactly one arc of~$T$. Call this edge $j$ and let $k$ denote the encircling edge. Given a general point of $X(\mathbb{S},\mathbb{M})$, we can define cross ratios $Y_j$ and~$Y_k$ as in part~(ii) of Definition~\ref{def:fockgoncharovcoordinates}. If we modify the framed local system by the action of the element of the group $\{\pm1\}^{\mathbb{P}}$ corresponding to the puncture~$p$, then by the proof of Lemma~12.3 in~\cite{IHES}, we know that the numbers $Y_j$ and~$Y_k$ transform to~$1/Y_j$ and~$Y_jY_k$, respectively. It now follows from Definition~\ref{def:fockgoncharovcoordinates} that the Fock-Goncharov co-ordinates associated to two signed arcs are equal if these signed arcs represent the same tagged arc.
\end{proof}

Note however that the inverse $X_\tau^{-1}$ of the map~\eqref{bira} may not be an open embedding for a general tagged triangulation~$\tau$.

\begin{example}
Suppose $\bS$ is a disc and $\bM=\{p_0,p_1,p_2\}$ where $p_0$ is a point in the interior of $\bS$ and $p_1$ and $p_2$ are marked points on $\partial\bS$. Consider the tagged triangulation $\tau$ that consists of an arc from~$p_0$ to each boundary marked point and the signing $\epsilon(p_0)=-1$. Then the inverse $X_\tau^{-1}$ is not regular at $(-1,-1)\in(\mathbb{C}^*)^2$. This follows from the proof of Lemma~12.3 in~\cite{IHES}, which gives an explicit expression in coordinates for the framing at $p_0$ after we act by the group element $\epsilon\in\{\pm1\}^{\mathbb{P}}$.
\end{example}

The signed mapping class group of a marked bordered surface $(\bS,\bM)$ is defined to be the semi-direct product
\[
\MCG^\pm(\bS,\bM)=\MCG(\bS,\bM)\ltimes \{\pm 1\}^{\bP}.
\]
This group acts birationally on the stack $\cX(\bS,\bM)$ of framed local systems and on the manifold $\Proj^*(\bS,\bM)$ in an obvious way, and our monodromy map $F$ is clearly equivariant. It also acts on the set of signed or tagged triangulations in the obvious way. By the above definition, one has
\[
X_{g(\tau)}(g(\mu))=X_\tau(\mu)
\]
for any element $g$ of this group.

\subsection{Behaviour under flips}

If $(T',\epsilon)$ is the signed triangulation obtained from $(T,\epsilon)$ by performing a flip at the edge~$k$, then the set $J=J^T$ of arcs in~$T$ is naturally in bijection with the set $J'=J^{T'}$ of arcs in~$T'$, and thus we can use the construction described above to associate a number $X_j'\in\mathbb{C}^*$ to each $j\in J'=J$. The following proposition relates these to the numbers $X_j$ computed using the signed triangulation $(T,\epsilon)$.

\begin{proposition}
The numbers $X_j'$ are given in terms of the numbers $X_j$~($j\in J$) by the formula
\begin{align*}
X_j'=
\begin{cases}
X_k^{-1} & \mbox{if } j=k \\
X_j{(1+X_k^{-\sgn(\varepsilon_{jk})})}^{-\varepsilon_{jk}} & \mbox{if } j\neq k.
\end{cases}
\end{align*}
\end{proposition}

\begin{proof}
The most nontrivial case to consider is when the flip produces a self-folded triangle as shown in Figure~\ref{fig:makeselffolded}. We will prove the formula in this case and leave the other cases to the reader.
\begin{figure}[ht]
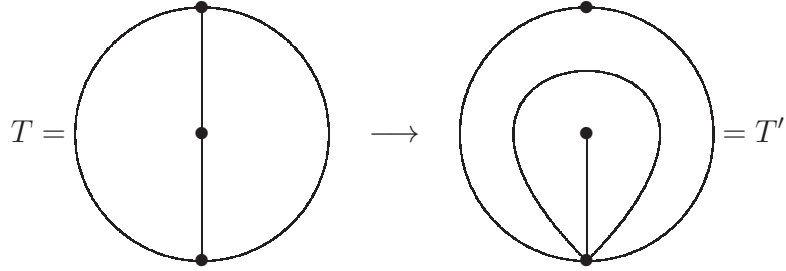
 \begin{center}
\[
T=
\xy 0;/r.50pc/:
(0,0)*\xycircle(8,8){-};
(0,-8)*{}="S";
(0,8)*{}="N";
"N";"S" **\dir{-};
(0,0)*{\bullet};
"N"*{\bullet};
"S"*{\bullet};
\endxy
\quad
\longrightarrow
\quad
\xy 0;/r.50pc/:
(0,0)*\xycircle(8,8){-};
(0,-8)*{}="S";
(0,8)*{}="N";
(0,0);"S" **\dir{-};
(0,0)*{\bullet};
"N"*{\bullet};
"S"*{\bullet};
"S";"S" **\crv{(-16,8) & (16,8)};
\endxy
=T'
\]
\caption{A flip resulting in a self-folded triangle.\label{fig:makeselffolded}}
\end{center} \end{figure}
We assume for simplicity that our surface is not a three-punctured sphere so that the edges on the sides of the above triangulation are not identified. Given a framed local system, we act by $\epsilon\in\{\pm1\}^{\mathbb{P}}$ to get a new framed local system and then let $Y_j$ be the cross ratio defined for any edge $j$ in~$T$ using this modified framed local system. These numbers transform as in Figure~\ref{fig:transformselffolded}.
\begin{figure}[ht]
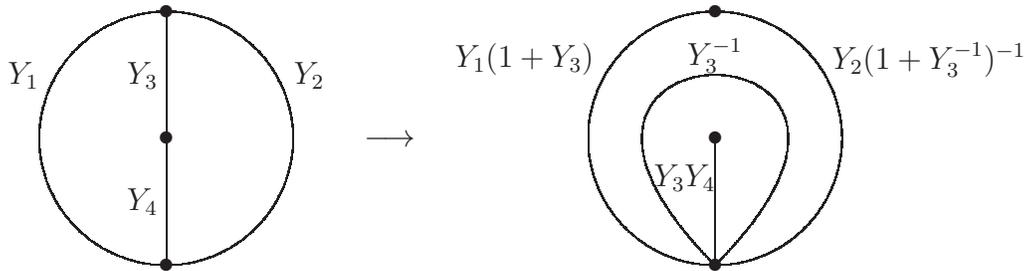
 \begin{center}
\[
{
\xy 0;/r.50pc/:
(0,0)*\xycircle(8,8){-};
(0,-8)*{}="S";
(0,8)*{}="N";
"N";"S" **\dir{-};
(0,0)*{\bullet};
(-1.5,4)*{Y_3};
(-1.5,-4)*{Y_4};
(-9,4)*{Y_1};
(9,4)*{Y_2};
"N"*{\bullet};
"S"*{\bullet};
\endxy
\quad
\longrightarrow
\quad
\xy 0;/r.50pc/:
(0,0)*\xycircle(8,8){-};
(0,-8)*{}="S";
(0,8)*{}="N";
(0,0);"S" **\dir{-};
(0,0)*{\bullet};
(0,5.25)*{Y_3^{-1}};
(-1.9,-2.5)*{Y_3Y_4};
(-12,5)*{Y_1(1+Y_3)};
(13.5,5)*{Y_2(1+Y_3^{-1})^{-1}};
"N"*{\bullet};
"S"*{\bullet};
"S";"S" **\crv{(-16,8) & (16,8)};
\endxy
}
\]
\caption{Transformation of cross ratios.\label{fig:transformselffolded}}
\end{center} \end{figure}
This implies the above formula by Definition~\ref{def:fockgoncharovcoordinates}. Thus we have proved the formula for flips resulting in a self-folded triangle, assuming the surface is not a three-punctured sphere. One can check the formula in this remaining case in a similar manner.
\end{proof}

\subsection{Regularity of points of $\cX^*(\bS,\bM)$}

A simple piece of combinatorics shows that at each point of $\cX^*(\bS,\bM)$ there is a tagged triangulation of $(\bS,\bM)$ such that the corresponding Fock-Goncharov co-ordinates are well-defined and nonzero at the given point.

\begin{theorem}
\label{clev}
Suppose that the framed local system $(\cG,\ell)$ is non-degenerate. Then there exists a signed triangulation such that the corresponding Fock-Goncharov co-ordinates are regular and non-zero at $(\cG,\ell)$. When $\bS$ has non-empty boundary the signing can be taken to be trivial. 
\end{theorem}

\begin{proof}
Let us call an edge of an ideal triangulation bad if the two framings associated to the ends of the edge agree under parallel transport along the edge. Otherwise we call the edge good. Condition~(D1) of Definition \ref{rn} ensures that a boundary segment is a good edge in any ideal triangulation. Let us first assume that there is an ideal triangulation with at least one good edge. This is automatic if $\bS$ has a non-empty boundary. We will then prove the existence of an ideal triangulation all of whose edges are  good.
 
Consider an ideal triangulation $T$ of $(\bS,\bM)$ with a maximal number of good arcs.  We need to show that in fact all arcs are good.  We call a triangle good if all of its edges are good, bad if exactly one of its edges is bad, and very bad if all of its edges are bad. We can assume the triangulation contains at least one bad or very bad triangle since otherwise we are done.

Suppose that a good triangle $t_1$ shares an edge $e_1$ with a bad triangle $t_2$. Then $e_1$ is good so there must be another triangle $t_3$ sharing a bad edge $e_2$ with $t_2$. Figure~\ref{fig:configurations} shows all possible configurations of the triangles $t_1,t_2,t_3$ up to reversal of orientation. The union $t_1\cup t_2\cup t_3$ can be lifted to a pentagon $p$ in the universal cover of $\mathbb{S}^*$. The framing lines determine at least three distinct lines associated to the vertices of~$p$ since the triangle $t_1$ is good. It is therefore possible to replace the edges $e_1$ and $e_2$, one of which is bad, by two good edges. This contradicts maximality.

\begin{figure}[ht]
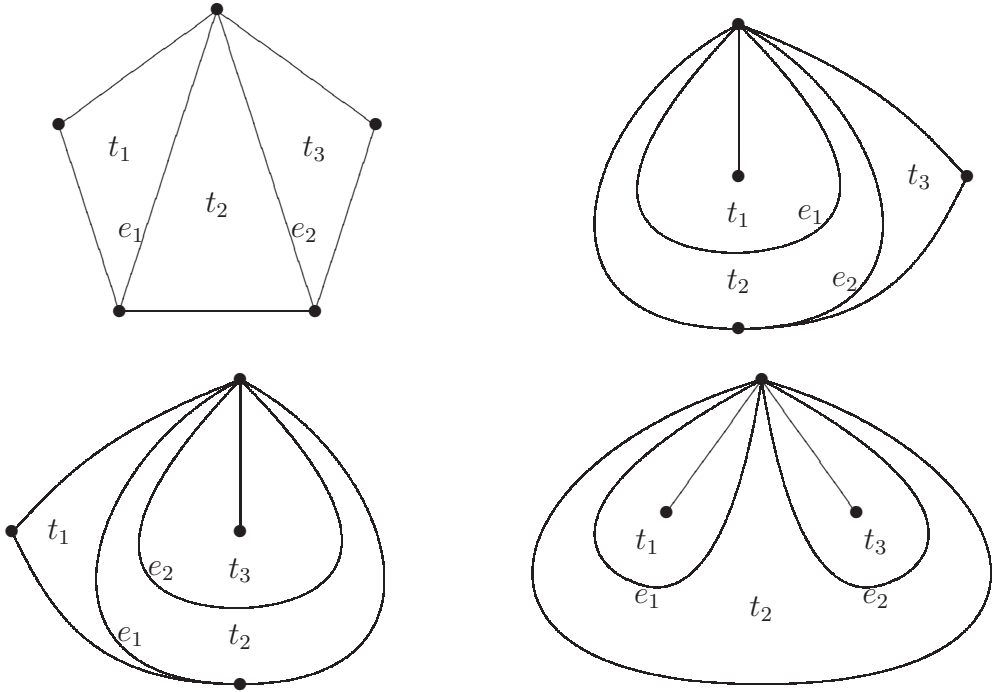
 \begin{center}
\[
\begin{array}{cc}
{\xy /l1.75pc/:
{\xypolygon5"A"{~:{(-3,0):}}};
{"A5"\PATH~={**@{-}}'"A2"'"A4"};
"A1"*{\bullet};
"A2"*{\bullet};
"A3"*{\bullet};
"A4"*{\bullet};
"A5"*{\bullet};
(2.75,-0.5)*{t_1};
(1,0.5)*{t_2};
(-0.75,-0.5)*{t_3};
(2.55,1)*{e_1};
(-0.57,1)*{e_2};
\endxy}
&
{\xy /l0.6pc/:
(0,4)*{}="M";
(0,-8)*{}="N";
(0,8)*{}="S";
(-12,0)*{}="P";
(0,0);"N" **\dir{-};
(0,0)*{\bullet};
"N"*{\bullet};
"S"*{\bullet};
"P"*{\bullet};
"N";"M" **\crv{(-12,4) & (0,4)};
"N";"M" **\crv{(12,4.5) & (0,4)};
"N";"S" **\crv{(8,-4) & (12,8)};
"N";"S" **\crv{(-8,-4) & (-12,8)};
"N";"P" **\crv{(-8,-5) & (-10,-2)};
"S";"P" **\crv{(-8,8) & (-10,4)};
(0,2)*{t_1};
(0,5.5)*{t_2};
(-9.5,0)*{t_3};
(-3.8,2)*{e_1};
(-5.6,5.5)*{e_2};
\endxy}
\\
\\
{\xy /l0.6pc/:
(0,4)*{}="M";
(0,-8)*{}="N";
(0,8)*{}="S";
(12,0)*{}="P";
(0,0);"N" **\dir{-};
(0,0)*{\bullet};
"N"*{\bullet};
"S"*{\bullet};
"P"*{\bullet};
"N";"M" **\crv{(-12,4) & (0,4)};
"N";"M" **\crv{(12,4.5) & (0,4)};
"N";"S" **\crv{(8,-4) & (12,8)};
"N";"S" **\crv{(-8,-4) & (-12,8)};
"N";"P" **\crv{(8,-5) & (10,-2)};
"S";"P" **\crv{(8,8) & (10,4)};
(0,2)*{t_3};
(0,5.5)*{t_2};
(9.5,0)*{t_1};
(4.15,2)*{e_2};
(5.75,5.5)*{e_1};
\endxy}
&
{\xy /l0.6pc/:
(5,-1)*{}="L1";
(7,2.5)*{}="L2";
(-5,-1)*{}="R1";
(-7,2.5)*{}="R2";
(0,-8)*{}="N";
(0,8)*{}="S";
(12,0)*{}="P";
"L1";"N" **\dir{-};
"L1"*{\bullet};
"R1";"N" **\dir{-};
"R1"*{\bullet};
"N"*{\bullet};
"N";"L2" **\crv{(2,5) & (6,3)};
"N";"L2" **\crv{(15,0) & (6,3)};
"N";"R2" **\crv{(-2,5) & (-6,3)};
"N";"R2" **\crv{(-15,0) & (-6,3)};
"N";"S" **\crv{(14,-4) & (18,8)};
"N";"S" **\crv{(-14,-4) & (-18,8)};
(6,0.5)*{t_1};
(0,4)*{t_2};
(-6,0.5)*{t_3};
(6,3.5)*{e_1};
(-6,3.5)*{e_2};
\endxy}
\end{array}
\]
\caption{Some possible configurations.\label{fig:configurations}}
\end{center} \end{figure}

Similarly note that a bad triangle cannot share an internal edge with a very bad triangle. Indeed, such an edge would have to be bad, and flipping it would lead to a triangulation with fewer bad edges. It is obvious that a good triangle cannot share an edge with a very bad triangle, and not all triangles can be very bad because we assumed that   the triangulation has at least one good edge. Thus we can assume that all triangles are bad.

Suppose that two bad triangles $t_1$, $t_2$ share an edge $e$. The union $t_1\cup t_2$ lifts to a quadrilateral $q$ in the universal cover of~$\mathbb{S}^*$, and the edge $e$ lifts to a diagonal of~$q$ which we will also denote by~$e$. Let us label the vertices of $q$ cyclically as $c_1,\cdots,c_4$ so that the edge $e$ connects $c_1$ and $c_3$, and the vertices of $t_1$ are $c_1,c_2,c_3$. The framings determine a line $\ell_i$ associated to each vertex $c_i$. We claim that there are exactly two distinct lines in the set $\{\ell_1,\cdots,\ell_4\}$. To see the claim, suppose first that $e$ is a good edge so that $\ell_1\neq \ell_3$. Then since the triangles $t_1$ and $t_2$ are bad we must have $\ell_2, \ell_4\in \{\ell_1,\ell_3\}$. Suppose instead that $e$ is bad so that $\ell_1=\ell_3$. Then if $\ell_2\neq \ell_4$ we can flip the edge $e$ to get a triangulation with fewer bad edges, contradicting maximality. The claim follows. 

We have now shown that parallel transporting the framing sections from the vertices of a triangle defines exactly two sections in the interior, and that for neighbouring triangles these sections coincide under parallel transport. It follows that the framed local system satisfies condition (D2) of Definition \ref{rn} which gives a contradiction.

We now return to the assumption that there is an ideal triangulation with at least one good edge. Suppose this fails. The surface $\bS$ is then necessarily closed. Choose an ideal triangulation. Since all edges are bad, it follows that there is a flat section $\ell$ of $\cL$ on $\bS^*$ such that each vertex is framed by~$\ell$. By non-degeneracy it follows that there is at least one puncture $p\in \bP$ such that the monodromy around $p$ is neither parabolic nor the identity.

Let us apply the corresponding element $\sigma_p$ of the group $\{\pm1\}^{\bP}$ to the framed local system. Then  the given triangulation has at least one good edge with respect to the new framed local system $\sigma_p(\cG,\ell)$. Note that by Remarks \ref{remas}(v) this flipped framed local system is still non-degenerate. Applying the argument above shows the existence of an ideal triangulation $T$ such that all Fock-Goncharov co-ordinates for $\sigma_p(\cG,\ell)$ are well-defined and non-zero. But then by definition, the co-ordinates for the framed local system $(\cG,\ell)$ with respect to the tagged  triangulation $\sigma_p(T)$ are also well-defined and non-zero.
\end{proof}

Since the maps \eqref{bira} are birational, one immediately obtains

\begin{cor}
For any marked bordered surface $(\bS,\bM)$ admitting an ideal triangulation the space  $\cX^*(\bS,\bM)$ is  a (possibly non-Hausdorff) complex manifold. 
\end{cor}

This completes the proof of Theorem \ref{one}, generalising the construction of~\cite{A}, Theorem~4.5 to arbitrary marked bordered surfaces.

\bibliographystyle{amsplain}

\begin{thebibliography}{100}
\bibitem{KoSc} Koike, T. and Sch\"afke, R. On the Borel summability of WKB solutions of Schr\"odinger equations with polynomial potentials and its application. To appear.
\end{thebibliography}

\end{document}